\documentclass[fleq]{amsart}
\usepackage{amssymb,amsmath,latexsym,amsbsy,color,enumerate}
\numberwithin{equation}{section}

\newtheorem{theorem}{Theorem}

\newtheorem{lemma}{Lemma}
\newtheorem{remark}{Remark}
\newtheorem{proposition}{Proposition}
\newtheorem{corollary}{Corollary}

\newtheorem{question}[theorem]{Question}

\begin{document}
\title[Salem numbers in complex dynamics on threefolds]
{Salem numbers in dynamics on K\"ahler threefolds 
and complex tori}
\author{Keiji Oguiso} 
 \author{Tuyen Trung Truong}
\address{Department of Mathematics, Osaka University, Toyonaka 560-0043, Osaka, Japan and Korea Institute for Advanced Study, Hoegiro 87, Seoul, 133-722, Korea}
 \email{oguiso@math.sci.osaka-u.ac.jp}   
       \address{Department of Mathematics, Syracuse University, Syracuse, NY 13244, USA}
 \email{tutruong@syr.edu}
 
\thanks{}
    \date{May 2013: preliminary version, September 2013: polished version}
    \begin{abstract} Let $X$ be a compact K\"ahler manifold of dimension $k\leq 4$ and $f:X\rightarrow X$ a pseudo-automorphism. If the first dynamical degree $\lambda _1(f)$ is a Salem number, we show that either $\lambda _1(f)=\lambda _{k-1}(f)$ or $\lambda _1(f)^2=\lambda _{k-2}(f)$. In particular, if $\mbox{dim}(X)=3$ then $\lambda _1(f)=\lambda _2(f)$. We use this to show that if $X$ is a complex $3$-torus and $f$ is an automorphism of $X$ with $\lambda _1(f)>1$, then $f$ has a non-trivial equivariant holomorphic fibration if and only if $\lambda _1(f)$ is a Salem number. If $X$ is a complex $3$-torus having an automorphism $f$ with $\lambda _1(f)=\lambda _2(f)>1$ but is not a Salem number, then the Picard number of $X$ must be $0,3$ or $9$, and all these cases can be realized.
\end{abstract}
\maketitle
\section{Introduction}       
\label{SectionIntroduction} 

Given $X$ a compact K\"ahler manifold and $f:X\rightarrow X$ a dominant meromorphic map. It is interesting to see whether $f$ has a non-trivial invariant meromorphic fibration. More precisely, the following question was asked by many people 

Question 1: Let $X$ be a compact K\"ahler manifold and $f:X\rightarrow X$ a dominant meromorphic map. Is there a compact K\"ahler manifold $Y$ with $0<dim (Y)<dim (X)$, and dominant meromorphic maps $\pi :X\rightarrow Y$ and $g:Y\rightarrow Y$ such that $\pi \circ f=g\circ \pi$?

D.-Q. Zhang called a map having no non-trivial meromorphic fibration a primitive map. Diller-Favre \cite{diller-favre} classified Question 1 in the case $X$ is a surface and $f$ is bimeromorphic map. For a dominant meromorphic map $f:X\rightarrow X$, there are associated bimeromorphic invariants called dynamical degrees (see Subsection \ref{SubsectionDynamicalDegrees} for definition of this). By the works in \cite{dinh-nguyen} and \cite{dinh-nguyen-truong},  having a non-trivial invariant meromorphic fibration put some constraints on dynamical degrees of $f$. For example, if $X$ has dimension $3$ and $f$ is pseudo-automorphic (we recall here that $f$ is pseudo-automorphic if it is bimeromorphic and both $f$ and $f^{-1}$ have no exceptional divisors, see e.g. \cite{dolgachev-ortland}) then we necessarily have $\lambda _1(f)=\lambda _2(f)$, and if moreover $\pi$ is holomorphic and $dim(Y)=2$ then $\lambda _1(f)$ is a Salem number (see Lemma \ref{LemmaPseudoAutomorphisSalemNumber}). As a first step toward resolving Question 1, we may ask whether the reverse of this is true, that is:

Question 2: Let $X$ be a compact K\"ahler manifold of dimension $3$ and $f:X\rightarrow X$ a pseudo-automorphism with $\lambda _1(f)>1$.

i) If $\lambda _1(f)=\lambda _2(f)$, does $f$ have a non-trivial invariant meromorphic fibration?

ii) If $\lambda _1(f)$ is a Salem number, does $f$ have a non-trivial invariant meromorphic fibration?

Our first main result below shows that for pseudo-automorphisms $f$ in dimensions $3$ and $4$, that $\lambda _1(f)$ being a Salem number has an interesting dynamical consequence. 
\begin{theorem}
Let $f:X \rightarrow X$ be a pseudo-automorphism of a compact K\"ahler manifold. Assume that $\lambda _1(f)$ is a Salem number. 

1) If $\mbox{dim}(X)=4$  then either $\lambda _1(f)=\lambda _{3}(f)$ or $\lambda _1(f)^2=\lambda _{2}(f)$. 

2) If $\mbox{dim}(X)=3$ then $\lambda _1(f)=\lambda _2(f)$.
\label{TheoremSalemNumberAgain}\end{theorem} 
Note that since $f$ is pseudo-automorphic, $\lambda _1(f)$ is the spectral radius of  the linear map $f^*:H^{1,1}(X)\rightarrow H^{1,1}(X)$ (see Lemma \ref{LemmaPseudoAutomorphismCompatibility}). Since $f^*$ preserves the cone psef cohomology classes, it follows from Perron-Frobenius theorem that $\lambda _1(f)$ is an eigenvalue of $f^*:H^{1,1}(X)\rightarrow H^{1,1}(X)$.  

In connection with Question 2, Bedford-Kim \cite{bedford-kim4} constructed an example of a pseudo-automorphism  $ f: X \to  X$, where $X$ is a composition of point blowups of ${\mathbf P}^3$.  This map has no non-trivial invariant meromorphic fibration, but its dynamical degrees have the property that $\lambda_1(f) = \lambda_2(f) $ is a Salem number strictly greater than 1.  However,
our next main result shows that the answer to Question 2 ii) is affirmative for the case $X$ is a complex $3$-tori and $\pi$ is an equivariant holomorphic fibration. 
\begin{theorem}
Let $X$ be a complex $3$-torus and $f \in {\rm Aut}\, (X)$ with $\lambda _1(f)>1$. If $\lambda_1(f)$ is a Salem number, then $X$ admits a non-trivial $f$-equivariant holomorphic fibration .
\label{TheoremToriAgain}\end{theorem}
{\bf Remarks}.  

1) We will prove a stronger conclusion: $X$ admits $f$-equivariant holomorphic fibrations both over a complex $1$-torus $Y_1$ and complex $2$-torus $Y_2$ such that $\lambda_1(f) = \lambda_1(g_2)$ for the induced automorphism $g_2$ of $Y_2$.

2) The converse of Theorem \ref{TheoremToriAgain} is also true, see Corollary \ref{CorollaryNonTrivialFibrationDimension3}.

Our last two main results show that even in the category of complex $3$-tori equipped with automorphisms and equivariant holomorphic fibrations, the answer to Question 2 i) is negative. However the answer is affirmative if some  constraints on the Picard number $\rho$ of the torus $X$ are satisfied.
\begin{theorem}
1) There is a projective $3$-torus $X$ with Picard number $9$ and has an automorphism $f:X\rightarrow X$ with $\lambda _1(f)=\lambda _2(f)>1$ but $f$ has no non-trivial equivariant holomorphic fibrations.  

2) There are non-projective $3$-torus $X$ with Picard number $0$ or $3$ having an automorphism $f:X\rightarrow X$ with $\lambda _1(f)=\lambda _2(f)>1$ but $f$ has no non-trivial equivariant holomorphic fbrations.
\label{TheoremExamples}\end{theorem}
 
\begin{theorem}
Let $X$ be a complex $3$-tori having an automorphism $f:X\rightarrow X$ with $\lambda _1(f)=\lambda _2(f)>1$ but $f$ has no non-trivial equivariant holomorphic fibrations. 

1) If $X$ is projective then the Picard number of $X$ must be $9$. 

2) If $X$ is non-projective then the Picard number of $X$ must be $0$ or $3$. 
\label{TheoremPicardNumber}
\end{theorem}

Theorems \ref{TheoremExamples} and \ref{TheoremPicardNumber} are proved by analysing the Galois groups of the characteristic polynomials of such maps $f$, similar to the argument used by Voisin \cite{voisin}. We can construct all examples in Theorem \ref{TheoremExamples}.   

{\bf Acknowledgements.} The authors would like to thank Eric Bedford for his kindly sending us the preprint \cite{bedford-kim4}, and for his interest in the topic of the paper. Theorems \ref{TheoremExamples} and \ref{TheoremPicardNumber} are in response to some of his questions. We also would like to thank Paul Reschke for his interest in the paper, and to him and the referee for helpful comments.  

\section{Proof of Theorem \ref{TheoremSalemNumberAgain}}
In the first subsection we recall the definition and some properties of dynamical degrees of meromorphic maps. In the second subsection we prove Theorem \ref{TheoremSalemNumberAgain} together with some other properties of pseudo-automorphisms.

\subsection{Dynamical degrees}
\label{SubsectionDynamicalDegrees}

Let $f:X\rightarrow X$ be a dominant meromorphic map, here $X$ is a compact K\"ahler manifold. There is a well-defined pullback map $f^*:H^{p,p}(X)\rightarrow H^{p,p}(X)$. One way to define this pullback map is as follows. We choose $\Gamma$ to be a resolution of singularities of the graph of $f$, and let $\pi, g:\Gamma\rightarrow X$ be the induced holomorphic maps to the source and the target. If $\theta$ be a smooth closed $(p,p)$ form, then $g^*(\theta )$ is also a smooth closed $(p,p)$ form. Then the pushforward $\pi _*g^*(\theta )$ is well-defined as a closed $(p,p)$ current. The cohomology class of $\pi _*g^*(\theta )$ depends only on the cohomology class of $\theta$. Thus we define $f^*\{\theta\}:=\{\pi _*g^*(\theta)\}$. The pullback $f^*:H^{p,p}(X)\rightarrow H^{p,p}(X)$ is linear.  Let $\omega _X$ be a K\"ahler form on $X$. For $0\leq p\leq k= {\rm dim}, X$, the $p$-th dynamical degree of $f$ is defined by
\begin{eqnarray*}
\lambda _p(f)=\lim _{n\rightarrow\infty}||(f^n)^*(\omega _X^p)||^{1/n}.
\end{eqnarray*}
Dinh and Sibony (\cite{dinh-sibony10} and \cite{dinh-sibony1}), showed that the dynamical degrees are well-defined (i.e. the limits in the definition exist), and are bimeromorphic invariants.

There is one other method to define $\lambda _p(f)$. We let $r_p(f)$ $=$ $\max |\lambda |$, here $\lambda$ runs over all eigenvalues of $f^*:H^{p,p}\rightarrow H^{p,p}(X)$ be the spectral radius of this pullback. Then
\begin{eqnarray*}
\lambda _p(f)=\lim _{n\rightarrow\infty}r_p(f^n)^{1/n}.
\end{eqnarray*}
In particular, if $f$ is holomorphic then $\lambda _p(f)=r_p(f)$.

Some properties of dynamical degrees: $\lambda _p(f)\geq 1$ for all $p\geq 1$, and  $\lambda _p(f)^2\geq \lambda _{p-1}(f)\lambda _{p+1}(f)$ (log-concavity). 

When $f$ is holomorphic, the results by Gromov \cite{gromov} and Yomdin \cite{yomdin} prove that the topological entropy $h_{top}(f)$ of $f$ equals $\max _{0\leq p\leq k}\log \lambda _p(f)$. For a general meromorphic map, we can still define its topological entropy. Dinh and Sibony \cite{dinh-sibony1}, proved that $h_{top}(f)\leq \max _{0\leq p\leq k}\log \lambda _p(f)$.  

Given $0\leq p\leq k$. We say that $f^*$ is $p$-algebraic stable if for any $n\in \mathbb{N}$, $(f^n)^*=(f^*)^n$ as linear maps on $H^{p,p}(X)$ (see Fornaess-Sibony \cite{fornaess-sibony1}). We can define similar notion for the pushforward $f_*$. 

\subsection{Pseudo-automorphisms}

Let $X$ be a compact K\"ahler manifold of dimension $k$ and $f:X \rightarrow X$ be a pseudo-automorphism. 

The following result was given in Propositions 1.3 and 1.4 in \cite{bedford-kim} when $\mbox{dim}(X)=3$.  
\begin{lemma}
1) The maps $f_*$ and $f^*$ are all $1$- algebraic stable. 

2) $(f^{-1})^*=f_*:H^{1,1}(X)\rightarrow H^{1,1}(X)$ is the inverse of $f^*:H^{1,1}(X)\rightarrow H^{1,1}(X)$.

\label{LemmaPseudoAutomorphismCompatibility}\end{lemma}
\begin{proof}
1) Let $\theta$ be a closed smooth $(1,1)$ form. Then since $f$ is pseudo-automorphic, the two currents $(f^n)^*(\theta )$ and $(f^*)^n(\theta )$ differ only on an analytic set of codimension at least $2$. Therefore they are the same current. Taking cohomology classes, we obtain 1). 

2) Proving similarly as in 1), we have that $f^*f_*$ is the identity map on $H^{1,1}(X)$. Because $f$ is bimeromorphic, $f_*=(f^{-1})^*$. From this we obtain 2).  
\end{proof}

The following lemma is essentially known, but we present a proof for the convenience of the readers.  
\begin{lemma}
Let $X$ be a compact K\"ahler manifold of dimension $3$ and $f:X\rightarrow X$ a pseudo-automorphism. Assume that there is a compact K\"ahler manifold $Y$ with $0 < {\rm dim}\, Y <{\rm dim}\, X =3$, and dominant meromorphic maps $\pi :X\rightarrow Y$ and $g:Y\rightarrow Y$ such that $\pi \circ f=g\circ \pi$.

i) Then $\lambda _1(f)=\lambda _2(f)$.

ii) If in addition $\pi$ is holomorphic and ${\rm dim}\, Y =2$,  then $\lambda _1(f)$ is either $1$ or a Salem number. 
\label{LemmaPseudoAutomorphisSalemNumber}\end{lemma}
\begin{proof}

i) There are two cases to consider: 

Case 1: ${\rm dim}\, Y =2$. By results in Dinh-Nguyen \cite{dinh-nguyen} (for projective manifolds) and in \cite{dinh-nguyen-truong} (for general compact K\"ahler manifolds), in the case where the maps $\pi$, $f$ and $g$ are holomorphic see also  Nakayama-Zhang \cite{nakayama-zhang}, there are relative  dynamical degrees $\lambda _0(f|\pi )$ and $\lambda _1(f|\pi )$ with the following properties: $\lambda _0(f|\pi )=1$, $\lambda _1(f|\pi )\geq 1$, and
\begin{eqnarray*}
\lambda _3(f)&=&\lambda _2(g)\lambda _1(f|\pi ),\\
\lambda _2(f)&=&\max \{\lambda _1(g)\lambda _1(f|\pi ),\lambda _2(g)\lambda _0(f|\pi )\},\\
\lambda _1(f)&=&\max \{\lambda _1(g)\lambda _0(f|\pi ), \lambda _0(g)\lambda _1(f|\pi )\}.
\end{eqnarray*}
  
Because $f$ is a pseudo-automorphism, we have $\lambda _3(f)=1$. Since both $\lambda _2(g),\lambda _1(f|\pi )\geq 1$, it follows that $\lambda _2(f)=\lambda _1(f|\pi )=1$. Then it follows that $\lambda _2(f)=\lambda _1(f)=\lambda _1(g)$.

Case 2: ${\rm dim}\, Y =1$. In this case we proceed as in Case 1. 

ii) If $\pi$ is holomorphic and ${\rm dim}\, Y =2$ then $g$ is a pseudo-automorphism of a surface, hence must be an automorphism. Therefore $\lambda _1(g)$ is a Salem number (see e.g. Diller-Favre \cite{diller-favre}). From Case 1 of i) we find that $\lambda _1(f)=\lambda _1(g)$ is a Salem number.
\end{proof}

\subsection{Proof of Theorem \ref{TheoremSalemNumberAgain}} 

We now give the proof of Theorem \ref{TheoremSalemNumberAgain}. By part 2) of Lemma \ref{LemmaPseudoAutomorphismCompatibility}, if $\lambda$ is an eigenvalue of $f^*:H^{1,1}(X)\rightarrow H^{1,1}(X)$ then $1/\lambda$ is an eigenvalue of $(f^{-1})^*:H^{1,1}(X)\rightarrow H^{1,1}(X)$. Hence Theorem \ref{TheoremSalemNumberAgain} follows from the following stronger result.    

\begin{proposition} Assume that  $f:X\rightarrow X$ is bimeromorphic, $\lambda _1(f)>1$, and $1/\lambda _1(f)$ is also an eigenvalue of $f^*:H^{1,1}(X)\rightarrow H^{1,1}(X)$. 

1) If $\mbox{dim}(X)=4$ then either $\lambda _1(f)=\lambda _{3}(f)$ or $\lambda _1(f)^2=\lambda _{2}(f)$. 

2) If $\mbox{dim}(X)=3$ then $\lambda _1(f)=\lambda _2(f)$.
\label{stronger}\end{proposition} 

\begin{proof}

1) First, we show the inequality $\lambda _1(f)\leq \lambda _1(f^{-1})=\lambda _{3}(f)$. In fact, since by assumption $1/\lambda _1(f)$ is an eigenvalue of $f^*:H^{1,1}(X)\rightarrow H^{1,1}(X)$, it follows from the first paragraph of the proof that $\lambda _1(f)$ is an eigenvalue of $(f^{-1})^*:H^{1,1}(X)\rightarrow H^{1,1}(X)$. Hence $\lambda _1(f^{-1})\geq \lambda _1(f)$. 

Next we show that either  $\lambda _1(f^{-1})\leq \lambda _1(f)$ or $\lambda _1(f)^2=\lambda _{2}(f)$. {\it Here we can not argue as above since a priori we do not know that $1/\lambda _1(f^{-1})$ is also an eigenvalue of $(f^{-1})^*:H^{1,1}(X)\rightarrow H^{1,1}(X)$.}

We instead proceed as follows. Again $\lambda _1(f)$ is an eigenvalue of $(f^{-1})^*:H^{1,1}(X)\rightarrow H^{1,1}(X)$; in particular 
$$\lambda _1(f^{-1})\geq \lambda _1(f)\,\, .$$ 
Since $f$ is bimeromorphic and $X$ has dimension $4$, we have $\lambda _2(f)=\lambda _2(f^{-1})$. From the log-concavity of dynamical degrees, we obtain
\begin{eqnarray*}  
\lambda _1(f^{-1})^2\geq \lambda _1(f)^2\geq \lambda _{2}(f)=\lambda _2(f^{-1})\,\, .
\end{eqnarray*}
If $\lambda _1(f)^2=\lambda _2(f)$, the proof is complete. Otherwise, then $\lambda _1(f)^2>\lambda _2(f^{-1})$. Theorem 1 in \cite{truong3} applied to $f^{-1}$ shows that $(f^{-1})^*:H^{1,1}(X)\rightarrow H^{1,1}(X)$ has only one eigenvalue larger than $\sqrt{\lambda _{2}(f^{-1})}$.  From the above, both $\lambda _1(f^{-1})=\lambda _3(f)$ and $\lambda _1(f)$ are eigenvalues of $(f^{-1})^*:H^{1,1}(X)\rightarrow H^{1,1}(X)$ larger than $\sqrt{\lambda _{2}(f^{-1})}$. Therefore, we must have $\lambda _1(f)=\lambda _{3}(f)$, as claimed. 

2) Proceed as in 1), we deduce that either $\lambda _1(f)^2=\lambda _2(f^{-1})$ or $\lambda _1(f)=\lambda _1(f^{-1})$. Because $\mbox{dim}(X)=3$, we have $\lambda _2(f^{-1})=\lambda _1(f)>1$. Therefore the first case can not happen, and we conclude that $\lambda _1(f)=\lambda _2(f)$.   
\end{proof}

\section{Equivariant holomorphic fibrations on complex tori}

Throughout this section, we consider a complex torus $X$ of ${\rm dim}\, X = n$, a {\it complex $n$-torus} for short. We write $X$ 
as the quotient of 
the universal covering space ${\mathbf C}^n$ by a discrete ${\mathbf Z}$-submodule $L \subset {\mathbf C}^n$,
 of rank $2n$: 
$$X = {\mathbf C}^n/L\, .$$ 
In this description, $L \simeq {\mathbf Z}^{2n}$ and we have a natural identification 
$${\mathbf C}^n = T_{X, 0} = H^0(X, \Omega_X^1)^*\, , \, 
L = \pi_1(X) = H_1(X, {\mathbf Z}) = H^1(X, {\mathbf Z})^*\,\, ,$$ 
and the weight one Hodge structure
$$H_1(X, {\mathbf Z}) \otimes {\mathbf C} = H^0(X, \Omega_X^1)^* \oplus 
\overline{H^0(X, \Omega_X^1)^*}\, .$$ 
Here $T_{X, 0}$ is the tangent space of $X$ at the origin $0$, $H^0(X, \Omega_X^1)$ is the space of global holomorphic $1$-forms on $X$, $H^0(X, \Omega_X^1)^*$ is the dual vector space of $H^0(X, \Omega_X^1)$ and $\overline{H^0(X, \Omega_X^1)^*}$ is the complex conjugate of $H^0(X, \Omega_X^1)$ with respect to the real structure $H^1(X, {\mathbf Z}) \otimes {\mathbf R}$ of $H^1(X, {\mathbf Z}) \otimes {\mathbf C}$ 
(See for instance \cite[Chapter 2, Section 6]{griffiths-harris}).

Note that the bimeromorphic selfmap of $X$ 
is necessarily biholomorphic, i.e., the group of bimeromorphic selfmaps ${\rm Bir}\, (X)$ coincides with the group of biholomorphic selfmaps ${\rm Aut}\, (X)$. This is because $X$ has no rational curve as complex subvarieties. We denote by ${\rm Aut}_{\rm group}\, (X)$ the group of automorphisms of the complex Lie group $X$. 

Let $f \in {\rm Aut}\, (X)$. We are interested in the following:
\begin{question}
When does $X$ 
admit a non-trivial $f$-equivariant holomorphic fibration 
$\pi :X\rightarrow Y$, 
that is, a surjective holomorphic map onto a complex analytic {\it K\"ahler} space $Y$ with $0 < {\rm dim}\, Y < {\rm dim}\, X$ and with $g \in {\rm Aut}\, (Y)$ such that $\pi \circ f=g\circ \pi$?
\label{QuestionTori}\end{question} 
Recall that the Stein factorization 
map is a unique finite morphism. So, by taking the Stein factorization of $\pi$, we may further assume without loss of generality that $Y$ is normal and $\pi$ has connected fibers. 

{\it In this section, we will use the notations introduced here and always assume that a holomorphic fibration is surjective onto a normal base space and has connected fibers.}

\subsection{$f$-admissible submodules and $f$-equivariant fibrations.}

We call a $\mathbf Z$-submodule $M$ of $L$ {\it admissible} 
if $M$ is primitive, i.e., $L/M$ is torsion free, and $M$ carries the weight one sub-Hodge structure of $L$ in the sense that 
$$M \otimes {\mathbf C} = (M \otimes {\mathbf C} \cap H^0(X, \Omega_X^1)^*) \oplus \overline{(M \otimes {\mathbf C} \cap H^0(X, \Omega_X^1)^*)}\,\, .$$
Each admissible submodule $M$ defines a complex subtorus 
$$Z_M := (M \otimes {\mathbf C} \cap H^1(X, \Omega_X^1)^*)/M$$ 
of $X$ and vice versa. Here by a complex {\it subtorus}, we mean a compact connected complex {\it Lie-subgroup} of the complex Lie group $X$. The primitivity in the converse is as follows. The quotient $X/Z$ by a complex subtorus $Z$ is a compact connected abelian holomorphic Lie group, i.e., a complex torus, and the quotient map $g : X \rightarrow X/Z$ is a locally trivial fibration. Therefore, we have a natural exact sequence
$$\pi_1(Z) \rightarrow \pi_1(X) \rightarrow \pi_1(X/Z) \rightarrow 0$$
in which $\pi_1(X/Z)$ is torsion free. 

The following important lemma is well-known. It is a direct consequence of Lemma 2.1 in \cite{demailly-hwang-peternell}.

\begin{lemma}
Any holomorphic fibration $\pi : X \rightarrow Y$ is the quotient map $X \rightarrow X/Z$ by a suitable subtorus $Z$ of $X$. 
\label{Ueno}\end{lemma}

\begin{proof} Let $Z$ be a general fiber of $\pi$. By Lemma 2.1 in \cite{demailly-hwang-peternell}, $Z$ is a subtorus of $X$, $\pi : X \rightarrow Y$ factors through $g : X/Z \rightarrow Y$ and $g$ is a finite map. Since $\pi$ has connected fibers, $g$ is then a bimeromorphic finite map. Since $Y$ is normal, this implies $g$ is an isomorphism by the Zariski main theorem.
\end{proof}

We call an admissible ${\mathbf Z}$-submodule $M$ of $L$ $f$-{\it admissible} 
if $f_*(M) = M$ 
and $0 < {\rm rank}\, M < 2n = {\rm rank}\, L$. 
Here $f_*$ is the automorphism of $L$ naturally induced by $f \in {\rm Aut}\, (X)$.

From this lemma, we obtain the following useful:

\begin{theorem}
The following conditions are equivalent: 

(i) $X$ admits an $f$-equivariant holomorphic fibration.

(ii) $L$ has an $f$-admissible ${\mathbf Z}$-submodule. 
\label{admissible}\end{theorem}

\begin{proof} Assume (i). Let $\pi : X \rightarrow Y$ be an $f$-equivariant holomorphic fibration. By Lemma \ref{Ueno}, $Y$ is a complex torus. Then 
$$M := {\rm Ker}\, 
(\pi_* : H_1(X, {\mathbf Z}) \rightarrow H_1(Y, {\mathbf Z}))$$ 
satisfies the requirement of (ii), by $f \circ \pi = \pi \circ g$. 
Note that $\pi_*$ is a homomorphism preserving weight one Hodge structure.

Assume (ii). Let $Z$ be a subtorus of $X$ corresponding to an $f$-admissible 
${\mathbf Z}$-submodule $M$ of $L$ and $q : X \rightarrow Y$ be the quotient map onto the quotient torus $Y = X/Z$. Decompose 
$f$ as $f = t_a \circ h$, where $h \in {\rm Aut}_{\rm group}\, (X)$ 
and $t_a : x \mapsto x+a$ is a translation by $a \in X$. Then $h(Z) =Z$ by the construction.
Thus 
$$f(Z +x) = h(Z +x) +a = h(Z) +h(x) +a = Z +(h(x)+a)$$
for all $x \in X$. Moreover, by $h(Z) = Z$, we have
$$h(x+z) - h(x) = h(z) \in Z$$
for all $z \in Z$ and $x \in X$. Hence the selfmap $f : x \mapsto h(x) +a$ of $X$ descends to the well-defined selfmap $g$ of $Y$ such that $\pi \circ f = g \circ \pi$. Since $\pi$ is a smooth surjective morphism and $Y$ is also smooth, it follows that $g \in {\rm Aut}\, (Y)$. Hence $g : X \rightarrow Y$ is an $f$-equivariant fibration. 
\end{proof}

\subsection{$f$-admissible submodules and the minimal polynomial of $f$.}

Consider the minimal polynomial $m(t)$ and the characteristic polynomial 
$\Phi(t)$ of the action $f_* \in {\rm Aut}\, (L)$ 
of $f \in {\rm Aut}\, (X)$. Then $m(t)$, $\Phi(t)$ are monic polynomials in ${\mathbf Z}[t]$. We have $f_* \otimes id_{\mathbf C} = F \oplus \overline{F}$, where 
$$F = f_* \otimes id_{\mathbf C} \vert H^0(X, \Omega_X^1)^*$$ and 
$\overline{F}$ is the complex conjugate map on 
$\overline{H^0(X, \Omega_X^1)^*}$ :
$$\overline{F} = f_* \otimes id_{\mathbf C} \vert \overline{H^0(X, \Omega_X^1)^*} = \overline{f_* \otimes id_{\mathbf C} \vert H^0(X, \Omega_X^1)^*}\,\, .$$ 
We have the following generalized eigenspace decomposition of the action of $f_*$ on $H^0(X, \Omega_X^1)^*$ and $\overline{H^0(X, \Omega_X^1)^*}$:
$$H^0(X, \Omega_X^1)^* = GV_F(\alpha_1) \oplus \cdots \oplus GV_F(\alpha_s)
\,\, ,$$
$$\overline{H^0(X, \Omega_X^1)^*} = 
\overline{GV_F(\alpha_1)} \oplus \cdots \oplus 
\overline{GV_F(\alpha_s)} = 
GV_{\overline{F}}(\overline{\alpha_1}) \oplus \cdots \oplus GV_{\overline{F}}(\overline{\alpha_s})\,\, ,$$
where $GV_F(\alpha)$ is the genaralized eigenspace of $F$ with eigenvalue $\alpha$ and similarly for $GV_{\overline{F}}(\alpha)$. 
We have then the $f$-stable decomposition
$$L \otimes {\mathbf C} = GV_F(\alpha_1) \oplus \cdots \oplus GV_F(\alpha_s) \oplus \overline{GV_F(\alpha_1)} \oplus \cdots \oplus 
\overline{GV_F(\alpha_s)}\,\, .$$
Note that $GV_F(\alpha_i) \oplus \overline{GV_F(\alpha_i)}$ is the generalized 
eigenspace of the action $f_*$ on $L \otimes {\mathbf C}$ if $\alpha_i$ 
is a real number. 

\begin{theorem}
If $m(t)$ is not irreducible in ${\mathbf Z}[t]$, that is, if there are $k(t), l(t) \in {\mathbf Z}[t] \setminus {\mathbf Z}$ such that $m(t) = k(t)l(t)$, then $X$ has an $f$-equivariant holomorphic fibration.  
\label{decomposition}\end{theorem}

\begin{proof}
Since $m(t)$ is monic and not irreducible in ${\mathbf Z}[t]$, there are two cases:

{\it Case (a).} $m(t)$ has at least two distinct irreducible factors over ${\mathbf Z}[t]$, i.e., there is a decomposition $m(t) = m_1(t)m_2(t)$ such that $m_1(t), m_2(t) \in {\mathbf Z}[t] \setminus {\mathbf Z}$ and $m_1(t)$ and $m_2(t)$ has no common root in ${\mathbf C}$. Recall that the roots of $m(t)$ and $\Phi(t)$ are the same modulo multiplicities. Therefore, we have also a decomposition $\Phi(t) = F_1(t)F_2(t)$ such that $F_1(t), F_2(t) \in {\mathbf Z}[t] \setminus {\mathbf Z}$ and $F_1(t)$ and $F_2(t)$ has no common root in ${\mathbf C}$.

{\it Case (b).} $m(t) = q(t)^k$ where $k \ge 2$ and $q(t) \in {\mathbf Z}[t]$ is irreducible.

First, consider Case (a). In this case, we shall show that the following stronger result which we shall use later:
\begin{proposition}
In Case (a), $X$ has at least two $f$-equivariant holomorphic fibrations 
$\phi_i : X \rightarrow Y_i$ ($i=1$, $2$) such that ${\rm dim}\, Y_1 + {\rm dim}\, Y_2 = {\rm dim}\, X$ and the characteristic polynomial of $(g_i)_* \vert H_1(Y_i, {\mathbf Z})$ is exactly $F_i(t)$. 
Here $g_i$ is the induced automorphism of $Y_i$. In particular, both ${\rm deg}\, F_i(t)$ are even. 
\label{twofibrations}\end{proposition}

\begin{proof} By the assumption, there are $h_1(t), h_2(t) \in {\mathbf Z}[t]$ and an integer $N$ such that 
such that 
$$F_1(t)h_1(t) + F_2(t)h_2(t) = N\,\, .$$ 
Consider the following ${\mathbf Z}$-submodules of $L$ corresponding to the factors $F_i(t)$ ($i=1$, $2$):
$$M_i' := F_{3-i}(f)h_{3-i}(f)(L)\,\, ,\,\, M_i := (M_i')^c\, ,$$
where $(M_i')^c$ is the primitive closure of the ${\mathbf Z}$-submodule $M_i' \subset L$, 
i.e., the smallest ${\mathbf Z}$-submodule such that $M_i' \subset (M_i')^c \subset L$ 
and $L/(M_i')^c$ is torsion free. In other words, $(M_i')^c = M_i' \otimes {\mathbf Q} \cap L$ in $L \otimes {\mathbf Q}$. By definition, $M_i$ are $f$-stable. Thus, by the basic property of the generalized eigenspace decomposition, $M_i \otimes {\mathbf C}$ is the direct sum of 
the generalized eigenspaces of $f_*$ with eigenvalue $\alpha_j$ such that 
$F_{i}(\alpha_j) = 0$. Note that $F_{i}(\alpha_j) = 0$ if and only if $F_{i}(\overline{\alpha_j}) = 0$. This is because 
$F_{i}(t) \in {\mathbf Z}[t]$. Hence, $\{1, 2, \cdots, s\}$ is decomposed into the disjoint union of two proper subsets $J_i$ ($i=1$, $2$) such that
$$M_i \otimes {\mathbf C} = \oplus_{j \in J_i} (VG_F(\alpha_j) \oplus VG_{\overline{F}}(\overline{\alpha_j})) = \oplus_{j \in J_i} VG_F(\alpha_j) \oplus \overline{\oplus_{j \in J_i} VG_F(\alpha_j)}$$
regardless that $\alpha_j$ is real or not (see also the remark before Theorem \ref{decomposition}). Hence, $M_i$ are $f$-admissible and therefore $X$ has 
two $f$-equivariant holomorphic fibrations $\pi_i : X \rightarrow Y_i := X/Z_{M_i}$ by Theorem \ref{admissible}. By construction, 
$${\rm dim}\, Y_i = {\rm dim}\, X - \frac{1}{2}{\rm rank}\,M_i\,\, ,\,\, 
{\rm rank}\, M_1\, +\, {\rm rank}\, M_2\, =\, 2{\rm dim}\, X\,\, .$$
Hence ${\rm dim}\, Y_1 + {\rm dim}\, Y_2 = {\rm dim}\, X$. The last two assertions follow from the generalized eigenspace decomposition above.  
\end{proof}

So, we are done in Case (a). Next, consider Case (b). Let $V_F(\alpha_i)$ (resp. $V_{\overline{F}}(\overline{\alpha_i})$) be the eigenspace of $F$ (resp. $\overline{F}$) with eigenvalue $\alpha_i$ (resp. $\overline{\alpha_i}$). Consider
$$M := {\rm Ker}\, (q(f) : L \rightarrow L)$$
and 
$$V := \oplus_{i=1}^{s} V_{F}(\alpha_i) \subset H^0(X, \Omega_X^1)^*\,\, ,\,\, 
\overline{V} = \overline{\oplus_{i=1}^{s} V_{F}(\alpha_i)} = \oplus_{i=1}^{s} V_{\overline{F}}(\overline{\alpha_i}) \subset \overline{H^0(X, \Omega_X^1)^*}\,\, .$$ 
Then $M$ is primitive and $f$-stable. Since $q(t) \in {\mathbf Z}[t]$, it follows from Gauss's elimination theory or the freeness, hence flatness, of ${\mathbf C}$ as ${\mathbf Z}$-module, we have
$$M \otimes {\mathbf C} = {\rm Ker}\, (q(f) \otimes id_{\mathbf C} : L \otimes {\mathbf C} \rightarrow L \otimes {\mathbf C})\,\, .$$
Now, by considering the Jordan canonical form of $f$, we obtain that
$$M \otimes {\mathbf C} = V \oplus \overline{V}$$
and therefore $M$ is $f$-admissible if $0 < {\rm rank}\, M < {\rm rank}\, L$. Since $\alpha_1$ is of multiplicity greater than or equal to $2$, it follows that $V_F(\alpha_1) \not= VG_F(\alpha_1)$. 
Hence $0 < {\rm rank}\, M < {\rm rank}\, L$ as desired. Therefore $X$ has 
an $f$-equivariant holomorphic fibration by Theorem \ref{admissible}. 
\end{proof}

\begin{remark}
Obviously, the converse of Theorem \ref{decomposition} is not true. For instance, consider $f := -id_{X \times X} \in {\rm Aut}\,(X \times X)$. Then, $m(t) = t+1$ but the second projection $p_2 : X \times X \rightarrow X$ is an $f$-equivariant holomorphic fibration. Here we notice that $-1$ is the eigenvalue of $f_* \vert H^0(X \times X, \Omega_{X \times X}^1)^*$ of multiplicity $2\cdot{\rm dim}\, X 
\ge 2$. 
\label{irreducible}\end{remark}

The following is a partial converse of Theorem \ref{decomposition}. 
Recall that $m(t)$ is the minimal polynomial $f_* \vert L$ over ${\mathbf Z}$.

\begin{proposition}
Assume that

(i) $m(t)$ is irreducible over ${\mathbf Z}[t]$ and;

(ii) $F = f_* \vert H^0(X, \Omega_X^1)^*$ has $n = {\rm dim}\, X = {\rm dim}\, H^0(X, \Omega_X^1)^*$ mutually distinct eigenvalues. 

Then $X$ has no $f$-equivariant holomorphic fibration.
\label{twiceminimal}\end{proposition}

\begin{proof} By (i), $m(t)$ has $n$ distinct simple roots. Denote them by $\alpha_i$ ($1 \le i \le n$).  
By (iii), the action $F = f_* \vert H^0(X, \Omega_X^1)$ has the following eigenspace decomposition:
$$H^0(X, \Omega_X^1)^* = \oplus_{i=1}^{n} V_{F}(\alpha_i)$$
and therefore $\overline{F}$-eigenspace decomposition
$$\overline{H^0(X, \Omega_X^1)^*} = \oplus_{i=1}^{n} V_{\overline{F}}(\overline{\alpha_i}) = \oplus_{i=1}^{n} \overline{V_{F}(\alpha_i)}\, .
$$
in which ${\rm dim}\, V(\alpha_i) = 1$ for each $i$ by ${\rm dim}\, H^0(X, \Omega_X^1) = n$. Note that the set 
$${\mathcal S} := \{\alpha_i, \overline{\alpha_i}\, \vert\, 1 \le i \le n\} \,\, ,$$
coincides with the set of roots of $m(t)$. This follows since $m(t)$ is the minimal polynomial of $f_* \otimes id_{\mathbf C} = F \oplus \overline{F}$ also over ${\mathbf C}$. 
Assume that there would be an $f$-admissible submodule $M$ of $L$. By ${\rm dim}\, V(\alpha_i) = 1$ and $f$-admissibility, there would be a non-empty subset $I$ of $\{1,2, \cdots , n\}$ such that
$$M \otimes {\mathbf C} = \oplus_{i \in I} (V_F(\alpha_i) \oplus {V_{\overline{F}}(\overline{\alpha_i}))}\,\, ,$$
regardless that $\alpha_i$ is real or not. 
Since $m(t)$ is irreducible over ${\mathbf Z}$, it follows that any two elements of ${\mathcal S}$ are Galois conjugate to each other over ${\mathbf Q}$. Thus, by $I \not= \emptyset$, we would have
$$M \otimes {\mathbf C} = \oplus_{i=1}^{n} (V_F(\alpha_i) \oplus V_{\overline{F}}(\overline{\alpha_i}))\,\, ,$$
whence $M = L$, by the primitivity, a contradiction. 
Hence there is no $f$-admissible ${\mathbf Z}$-submodule of $L$, and therefore by Theorem \ref{admissible}, $X$ admits no $f$-equivariant holomorphic fibration. 
\end{proof}

Proposition \ref{twiceminimal} can be used to construct a pair $(X, f \in {\rm Aut}\, (X))$ with no $f$-equivariant holomorphic fibration when ${\rm dim}\, X = 2k \ge 4$, even if $\lambda_1(f)$ is a Salem number. We need the following result on Salem numbers

\begin{lemma}
If $\alpha$ is a Salem number, then so is $\alpha^k$ for any positive 
integer $k$.
\label{power}\end{lemma}

\begin{proof} This follows that any Galois conjugate of $\alpha^k >1$ 
is $\beta^k$ for some Galois conjugate $\beta$ of $\alpha$ and vice versa. 
\end{proof}

\begin{lemma}
For each even positive integer $2k$, there is a Salem number of degree $2k$.
\label{gm}\end{lemma}

\begin{proof} There are Salem numbers of small degrees, say, $2$, $4$, $6$, $8$. When $k \ge 4$ is odd, this lemma is proved by Gross-McMullen (see Theorem 7.3 in \cite{gross-mcmullen}). They show first that for each even degree $2m$, there is a monic polynomial $C_{2m}(t) \in {\mathbf Z}[t]$ such that zeros of $C_{2m}(t)$ are all distinct, real and in $(-2, 2)$.  Then they show next that for odd $k \ge 5$, the polynomial 
$$R_{k}(t) = C_{k-3}(t)(t^2-4)(t-a) -1$$ is a Salem trace polynomial of degree $k$ for all large integer $a$ and more. This means that the polynomial $R_k(t)$ has $k-1$ roots belong to $(-2,2)$ and the other root is $>2$, and moreover $R_k(t)$ is irreducible. Let us briefly recall their argument.

The former is easy to show by inspecting the graph of the functions $C_{k-3}(t^2-4)$ and $1/(t-a)$. For the latter, we observe that if $P(t)$ is an irreducible factor of $R_k(t)$ having all roots in $(-2,2)$ then $P(t)$ belongs to a finite set of polynomials. Since when $a\rightarrow \infty$ the roots of $R_k(t)$ converge to the roots of $C_{k-3}(t)(t^2-4)$, eventually $P(t)$ must divide $C_{k-3}(t)$. But $R_k(t)=1$ at the roots of $C_{k-3}(t)$ and hence they can not have a nontrivial common factor. Therefore no such $P(t)$ exists, and hence $R_{k}(t)$ is irreducible because it has only one root outside of $(-2,2)$.

The corresponding Salem polynomial $S_{2k}(t)=t^kR_k(t+1/t)$ gives a desired Salem number. The same proof shows that for even $k$, the polynomial 
$$R_{k}(t) = C_{k-2}(t)(x-2)(x-a) -1$$ 
is a Salem trace polynomial of degree $k$ for large integer $a$. So, again, the corresponding Salem polynomial $S_{2k}(t)$ gives a desired Salem 
number. 
\end{proof}
  
Now  we proceed to construct our examples. 

\begin{proposition}
For each integer $k \ge 2$, there are a projective complex $2k$-torus $X$ 
and $f \in {\rm Aut}\, (X)$ such that 

(i) $\lambda _1(f)=\lambda _2(f)=\ldots =\lambda _{k-1}(f)$ and it is 
a Salem number, yet

(ii) $X$ has no $f$-equivariant holomorphic fibration. 
\label{EvenDimensionTori}\end{proposition}

\begin{proof}
Let $E$ be an elliptic curve and consider $X = E^{2k}$. Then, we have a natural inclusion ${\rm GL}_{2k}({\mathbf Z}) \subset {\rm Aut}_{\rm group}(X)$. 
Note that this is not true for general complex torus. 

Choose a Salem number $\alpha$ of degree $2k$. Such number exists by Lemma \ref{gm}. Let
$$S_{2k}(t) = \sum_{i=0}^{2k} a_{2k-i}t^i \in {\mathbf Z}[t]$$ 
be the minimal polynomial of $\alpha$. Consider the following matrix
$$M_{2k} = \left(\begin{array}{rrrrrrr}
0 & 1 & 0 & \ldots &0 &  0  & 0 \\
0 & 0  & 1 & \ldots &0&0  & 0 \\
\ldots & \ldots  & \ldots & \ddots &\ldots&\ldots  &\ldots \\
\vdots & \vdots & \vdots &\vdots& 1 & \vdots & \vdots \\ 
0 & 0 & \ldots&\ldots & 0 & 1 & 0 \\
0 & 0 & \ldots&\ldots & 0 & 0 & 1\\
-a_{2k} & -a_{2k-1} & \ldots&\ldots & -a_{3} & -a_{2} & -a_{1}\\
\end{array} \right)\,\, .$$
Then $M_{2k} \in {\rm SL}_{2k}({\mathbf Z})$ and the characteristic polynomial of $M_{2k}$, as well as its minimal polynomial, is the Salem polynomial $S_{2k}(t)$. Let $f$ be the corresponding automorphism of $X$. 
Then the action of $f$ on 
$H^0(X, \Omega_X^1)^*$ is given by the multiplication of $M_{2k}$ under the natural obvious dual basis of $H^0(X, \Omega_X^1)$. Thus $f$ satisfies all the requirements of Proposition \ref{twiceminimal} and the result (ii) follows. The dynamical degrees of $f$ satisfy 
$$\lambda _1(f)=\ldots =\lambda _{2k-1}(f) = \alpha^2\,\, .$$ 
This follows from $H^{2d}(X, {\mathbf Z}) = \wedge^{2d} H^1(X, {\mathbf Z})$ and the fact that $\alpha$ is a Salem number. Since $\alpha$ is a Salem number, so is $\alpha^2$ by Lemma \ref{power}. 
\end{proof}

\section{Automorphisms of complex $3$-tori with $\lambda _1(f) = \lambda_2(f) >1$}
\label{SectionFirstSecondDymanicalDegree}
 
{\it Throughout this section and next section, we consider a $3$-torus 
$$X = {\mathbf C}^3/L = H^1(X, \Omega_X^1)^*/H_1(X, {\mathbf Z})$$ 
and its automorphism $f \in {\rm Aut}\, (X)$ such that 
$$\lambda_1(f) = \lambda_2(f) >1\,\, .$$}

By Theorem \ref{TheoremSalemNumberAgain}, the last condition is satisfied if $\lambda_1(f)$ is a Salem number. We also use the following notations in this section and next section.

\subsection{Notation.}
As in the previous section, we naturally identify 
$$L = H_1(X, {\mathbf Z})\,\, ,\,\, {\mathbf C}^3 = H^0(X, \Omega_X^1)^*
\,\, .$$
We denote the eigenvalues of $f_*\vert  H^0(X, \Omega_X^1)^*$, counted with multiplicities, by 
$$\alpha, \beta, \gamma\,\, .$$
Without loss of generality, we may and will assume that 
$$\vert \alpha \vert \ge \vert \beta \vert \ge \vert \gamma \vert\,\, .$$
Then the set of eigenvalues of the ${\mathbf C}$-linear extension of $f_*\vert  H_1(X, {\mathbf Z}) = f_* \vert L$, counted with multiplicities, is 
$${\mathcal S}_1 := \{\alpha, \beta, \gamma, \overline{\alpha}, \overline{\beta}, \overline{\gamma}\,\}\,\, ,$$
and therefore the set of eigenvalues of $f_*\vert  H_2(X, {\mathbf Z})$, counted with multiplicities, is  
$${\mathcal S}_2 := \{ \alpha\overline{\alpha}, \alpha\overline{\beta}, \alpha\overline{\gamma}, 
\beta\overline{\alpha}, \beta\overline{\beta}, \beta\overline{\gamma}, 
\gamma\overline{\alpha}, \gamma\overline{\beta}, \gamma\overline{\gamma}, 
\alpha\beta, \alpha\gamma, \beta\gamma, 
\overline{\alpha}\overline{\beta}, \overline{\alpha}\overline{\gamma}, 
\overline{\beta}\overline{\gamma}\, \}\,\, .$$ 
This is because $H_2(X, {\mathbf Z}) = \wedge^2 H_1(X, {\mathbf Z})$, and 
therefore, ${\mathcal S}_2 = \wedge^2{\mathcal S}_1$ in an obvious sense. 
Similarly, the set ${\mathcal S}_i$ of eigenvalues of $f_* \vert H_i(X, {\mathbf Z})$ is $\wedge^i{\mathcal S}_1$. 

Under the natural identification $H_i(X, {\mathbf Z}) = H^i(X, {\mathbf Z})^*$, we have
$$f_* \vert H_i(X, {\mathbf Z}) = (f^* \vert H^i(X, {\mathbf Z}))^t\,\, .$$
Thus, the $k$-th $\lambda_k(f)$ is the spectral radius of $f_* \vert H_{2k}(X, {\mathbf Z})$. Under the same identification, one can speak of the Hodge decomposition of $H_i(X, {\mathbf Z})$, naturally induced by the standard one on $H^i(X, {\mathbf Z})$. 
Note however that under th natural identification 
$H_i(X, {\mathbf Z}) = H_{6-i}(X, {\mathbf Z})^*$ by the Poincar\`e duality, i.e., via the cup product, we have
$$f_* \vert H_i(X, {\mathbf Z}) = 
((f^* \vert H_{6-i}(X, {\mathbf Z}))^t)^{-1}\,\, .$$
We denote the characteristic polynomial of $f_*\vert L$ by
$$\Phi(t) := {\rm det}\, (t\cdot id_L -f) \in {\mathbf Z}[t]\,\, .$$
This is a monic polynomial of degree $6$ with integer coefficients. 
The set of roots of $\Phi(t)$, counted with multiplicities, is exactly ${\mathcal S}_1$, and the minimal splitting field of $\Phi(t)$ is 
$$K := {\mathbf Q}(\alpha, \beta, \gamma, \overline{\alpha}, \overline{\beta}, \overline{\gamma})\,\, \subset\,\, {\mathbf C}\,\, .$$
We denote the Galois group of $\Phi(t)$, i.e., the Galois group of the field extension $K/{\mathbf Q}$ by ${\rm G_f}$. ${\rm G_f}$ naturally acts on ${\mathcal S}_1$, ${\mathcal S}_2$. ${\rm G_f}$ also acts on 
$$H_i(X, K) = H_i(X, {\mathbf Z}) \otimes_{\mathbf Z} K$$
by $\sigma \mapsto id_{H_i(X, {\mathbf Z})} \otimes \sigma$. 
We call ${\rm G_f}x$ ($x \in K$) the Galois orbit of $x$ and an element in ${\rm G_f}x$ a Galois conjugate of $x$. 

Needless to say, the action of $f_* \vert H_i(X, {\mathbf Z})$ preserves the Hodge decomposition, but the action of ${\rm G_f}$ on $H_i(X, K)$ does not necessarily preserve the Hodge decomposition. In what follows, both the {\it geometric action} of $f_*$ and the {\it algebraic action} of ${\rm G_f}$ play crucial roles. 
\subsection{General properties.}
We shall use the following basic facts frequently:
\begin{lemma}
1) $\vert\alpha \beta \gamma\vert = 1$. In particular, the constant term of $\Phi(t)$ is $1$. 

2) $\lambda_1(f) = \vert\alpha\vert^2$ and $\lambda_2(f) = \vert\alpha\vert^2\vert\beta\vert^2$. 

3) $\alpha$, $\beta$, $\gamma$ are mutually distinct. More strongly
$$\vert \alpha \vert > \vert \beta \vert = 1 > \vert \gamma \vert = \frac{1}{\vert \alpha \vert}\,\, .$$

4) The $K$-linear extension of $f_* \vert H_i(X, {\mathbf Z})$ is diagonalizable, and so is ${\mathbf C}$-linear extension. 

5) If $X$ is projective, then $\alpha \beta \gamma$ is a root of unity.

\label{basic}\end{lemma}

\begin{proof} Since $f_* \vert L$ is invertible over ${\mathbf Z}$ and $L$ is of even rank, it follows that 
$$0 < \vert\alpha\beta\gamma\vert^2 = \alpha\beta\gamma\overline{\alpha\beta\gamma} = \Phi(0) = {\rm det}\, (f) \in \{\pm 1\}\,\, ,$$
and the first assertion of 1) follows. 

Note that $\alpha \beta \gamma$ is 
the eigenvalue of $f_* \vert H^0(X, \Omega_X^3)^*$. Then, by Proposition 14.5 in \cite{Ue}, it is a root of unity when $X$ is projective, and 5) holds. 

By $\vert\alpha\vert \ge \vert\beta\vert \ge \vert\gamma\vert$ and by the facts that ${\mathcal S}_2 = \wedge^2{\mathcal S}_1$ and ${\mathcal S}_4 = \wedge^4{\mathcal S}_1$, 2) follows. Since $\lambda_1(f) = \lambda_2(f) >1$, 
it follows from 2) that $\vert\beta\vert = 1$ and $\vert \alpha \vert >1$. 
Hence, by 1), $\vert \gamma \vert = 1/\vert \alpha \vert$, and 3) follows. 

By 3), the eigenvalues of $f_* \vert H^0(X, \Omega_X^1)^*$ are mutually distinct. It follows that $f_* \vert H^0(X, \Omega_X^1)^*$ is diagonalizable. Since the ${\mathbf C}$-linear extension of $f_* \vert L$ is 
$$f_* \vert H^0(X, \Omega_X^1)^* \oplus \overline{f_* \vert H^0(X, \Omega_X^1)^*}\,\, ,$$ 
it follows that $f_* \vert L$ is diagonalizable over ${\mathbf C}$. 
Since all the eigenvalues are in $K$, it is already diagonalizable over 
$K$. 
Hence $f_* \vert H_i(X, {\mathbf Z}) = \wedge^i f_*\vert L$ is diagonalizable also over $K$. 
\end{proof}

\begin{theorem}

1) If $\Phi(t)$ admits a decomposition $\Phi(t) = F_1(t)F_2(t)$ such that $F_1(t), F_2(t) \in {\mathbf Z}[t]$, $1 \le {\rm deg}\, F_1(t) \le {\rm deg}\, F_2(t)$ and $F_1(t)$, $F_2(t)$ have no common root in ${\mathbf C}$, then $X$ admits $f$-equivariant holomorphic fibrations both over a complex $1$-torus $Y_1$ and complex $2$-torus $Y_2$ such that the characteristic polynomial of $(g_i)_* \vert H_1(Y_i, {\mathbf Z})$ is exactly $F_i(t)$ ($i = 1, 2$). 
Here $g_i$ is the induced automorphism of $Y_i$. In particular, ${\rm deg}\, F_1(t) =2$ and ${\rm deg}\, F_2(t) = 4$. 

2) $X$ admits no $f$-equivariant holomorphic fibration if and only if  $\Phi(t)$ is irreducible in ${\mathbf Z}[t]$. Moreover, in this case, all roots of $\Phi(t)$ are non-real. 

\label{DoubleFibration}
\end{theorem} 

\begin{proof} 1) follows from Proposition \ref{twofibrations}. 
If $\Phi(t)$ is irreducible in ${\mathbf Z}[t]$, then it is also the minimal 
polynomial. Hence if part of 2) follows from Proposition \ref{twiceminimal} and Lemma \ref{basic} 3). If $X$ has no $f$-equivariant holomorphic fibration 
but $\Phi(t)$ would not be irreducible, then $\Phi(t) = F(t)^2$ for some irreducible $F(t) \in {\mathbf Z}[t]$ by 1) and Lemma \ref{basic} 3). However, then $\beta = \overline{\beta}$, whence $\beta = \pm 1$ again by Lemma \ref{basic} 3). Thus $F(t) = (t \pm 1) h(t)$ for some $h(t) \in {\mathbf Z}[t]$, a contradiction. 
So, if $X$ has no $f$-equivariant holomorphic fibration, then $\Phi(t)$ is irreducible in ${\mathbf Z}[t]$. In particular, the roots $\alpha$, $\beta$, $\gamma$ $\overline{\alpha}$, $\overline{\beta}$, $\overline{\gamma}$ are mutually distinct, whence, all are non-real. 
\end{proof}

\subsection{Proof of Theorem \ref{TheoremToriAgain}} 

In this subsection, we prove Theorem \ref{TheoremToriAgain} in the Introduction. We use the same notation as before. By the assumption of 
Theorem \ref{TheoremToriAgain}, $\lambda_1(f) = \vert\alpha\vert^2$ is a Salem number. We reduce the proof to Theorem \ref{DoubleFibration} by studying the action of the Galois group ${\rm G_f}$ on ${\mathcal S}_1$ and ${\mathcal S}_2$.

\begin{lemma}\label{galois}
$\beta$ is not a Galois conjugate of $\alpha$, i.e., there is no $\sigma \in {\rm G_f}$ such that $\beta = \sigma(\alpha)$.
\end{lemma}

\begin{proof} Since $\vert \alpha \vert^2$ is a Salem number, it follows from the description of the set ${\mathcal S}_2$ 
that $\vert \gamma \vert^2$ and $\vert \alpha \vert^2$ are Galois conjugates of $\vert \alpha \vert^2 = \alpha\overline{\alpha}$ and all other Galois conjugates $\delta$ of $\vert \alpha \vert^2$ satisfy $\delta \in {\mathcal S}_2$, $\vert \delta \vert = 1$ and $\delta \not\in {\mathbf Q}$. In particular, $\beta\overline{\beta} = 1$ is not a Galois conjugate of $\vert \alpha \vert^2$. 

{\it Assume to the contrary that there is $\sigma \in {\rm G_f}$ such that $\beta =\sigma ( \alpha )$}. Set $\epsilon := \sigma (\overline{\alpha})$ and 
$$x := {\sigma}(\alpha\overline{\alpha}) = {\sigma}(\alpha ) {\sigma}(\overline{\alpha}) = \beta\epsilon\, .$$ 
Then $x \in {\mathcal S}_2 \setminus \{\beta \overline{\beta} = 1\}$.

First consider the case where $\alpha \in {\mathbf R}$, i.e., the case where 
$\alpha = \overline{\alpha}$. Then $x = \beta^2$ and therefore 
$\vert x \vert = 1$. 
Thus $x$ is either $\alpha\gamma$ or $\alpha\overline{\gamma}$ by the description of ${\mathcal S}_2$ and the fact that $\alpha$ is real. So, we can write $x = \alpha y$, where $y$ is either $\gamma$ or $\overline{\gamma}$. Then ${\sigma}(x) = {\sigma}(\alpha y) = \beta  {\sigma}(y)$ is also Galois conjugate of the Salem number $\vert \alpha \vert^2$. Hence $\vert {\sigma}(y ) \vert$ is either $\vert \gamma \vert^2$, $\vert \alpha \vert^2$ or $1$. By $\alpha \not= y$, we have ${\sigma}(y) \in {\mathcal S}_1 \setminus \{\beta\}$. Hence by Lemma \ref{basic} 3), we have ${\sigma}(y) = \overline{\beta}$. However, then ${\sigma}(x) = \beta \overline{\beta} = 1$, 
a contradiction. 
 
Next consider the case where $\alpha \not\in {\mathbf R}$, i.e., the case 
where 
$\alpha \not= \overline{\alpha}$. Since ${\mathcal S}_1$ is preserved by 
${\rm G_f}$ and $\alpha \not= \overline{\alpha}$, it follows that $\epsilon = {\sigma}(\overline{\alpha}) \in {\mathcal S}_1 \setminus \{\beta\}$. Therefore, by the description of 
${\mathcal S}_2 \setminus \{\beta \overline{\beta} = 1\}$, our $x$ is either one of 
$$\beta\alpha\, ,\, \beta\gamma\, , \, \beta\overline{\alpha}\, ,\, \beta\overline{\gamma}\, .$$
However, then $\vert x \vert \not= 1$, $\vert x \vert \not= \vert \alpha \vert^2$ and $\vert x \vert \not= \vert \gamma \vert^2$ again by 
Lemma \ref{basic} 3), 
a contradiction to the fact that $x$ is a Galois conjugate of the Salem number $\vert \alpha \vert^2$. 

Hence $\beta$ is not a Galois conjugate of $\alpha$. 
\end{proof}

Hence we have a decomposition $\Phi(t) = F_1(t)F_2(t)$, $F_1(t), F_2(t) \in {\mathbf Z}[t] \setminus {\mathbf Z}$ such that $F_1(t)$ and $F_2(t)$ have no common root. Now Theorem \ref{TheoremToriAgain} follows from Theorem \ref{DoubleFibration}, 1). 

\begin{remark} By Proposition \ref{EvenDimensionTori}, a similar result does not hold for even dimensional complex tori. It is interesting to see whether a similar statement is true for higher odd dimensional complex tori $X$ or not. \label{HigherDimensionalOddTori}
\end{remark}

\subsection{Some consequences.} 

\begin{corollary}
Let $X$ be a complex $3$-torus with $f \in {\rm Aut}\, (X)$ such that 
$\lambda_1(f)>1$. Here we do NOT assume $\lambda_1(f) = \lambda_2(f)$ 
from the beginning. Nevertheless, the following conditions are equivalent: 

(i) $X$ has an $f$-equivariant holomorphic fibration.

(ii) $\lambda _1(f)$ is a Salem number.
\label{CorollaryNonTrivialFibrationDimension3}\end{corollary}
\begin{proof}

That (ii) implies (i) is proved in Theorem \ref{TheoremToriAgain}.

We prove that (i) implies (ii). Let $\pi : X \rightarrow Y$ be an $f$-equivariant holomorphic fibration and $g$ be the automorphism of $Y$ induced by $f$. 
We have two cases : 

{\it Case (a)} ${\rm dim}\, Y = 2$, {\it Case (b)} ${\rm dim}\, Y = 1$. 

{\it Case (a).} This follows from part ii) of Lemma \ref{LemmaPseudoAutomorphisSalemNumber}.  

{\it Case (b).} It suffices to show that $X$ has an $f$-equivariant holomorphic fibration also over complex $2$-torus. 

Let $\alpha$, $\beta$, $\gamma$ be the eigenvalues of $f_* \vert H^0(X, \Omega_X^1)^*$ counted with multiplicities. Since $Y$ is a complex $1$-torus, the eigenvalue of $g_*$ on $H^0(Y, \Omega_Y^1)^{*}$ are one of $\exp(2\pi i/k)$ with $k = 1, 2, 4, 6$. Their minimal polynomials are of degree $\le 2$ over $\mathbf Z$. It is also an eigenvalue of $f_* \vert H^0(X, \Omega_X^1)^*$, because the homomorphism $\pi_* : H_1(X, {\mathbf Z}) \rightarrow H_1(Y, {\mathbf Z})$ is $f$-equivariant, surjective and preserves Hodge structures of weight one. 

On the other hand, some eigenvalue, say $\alpha$, of $f_* \vert H^0(X, \Omega_X^1)^*$ satisfies $\vert \alpha \vert >1$. This is because $\lambda_1(f) >1$. In particular, $\alpha$ is not a Galois conjugate 
of $\beta$. Hence, the minimal polynomial $m(t) \in {\mathbf Z}[t]$ of $f_*$ 
satisfies Case (a) in the proof of Theorem \ref{decomposition}. Thus, by Proposition \ref{twofibrations}, we have an $f$-equivariant holomorphic fibration $X \rightarrow Y'$ also over a complex $2$-torus, as desired.
\end{proof}

The following corollary classifies the Salem first dynamical degrees of automorphisms of complex $3$-tori:

\begin{corollary}
Let ${\it Sal}(n)$ be the set of Salem first dynamical degrees of automorphisms of complex $n$-tori. Then ${\it Sal}(3) = {\it Sal}(2)$. 

\label{Corollary3TorusSalemNumber}\end{corollary}

\begin{remark} 
Note that the first dynamical degree of an automorphism of complex $2$-tori 
is a Salem number if and only if it is greater than $1$. Reschke 
gives a complete description of ${\it Sal}(2)$ in terms of the corresponding 
Salem polynomials (\cite{reschke}).
\label{Reschke}\end{remark}

\begin{proof}

Proof of ${\it Sal}(2) \subset  {\it Sal}(3)$: Let $E$ be a complex $1$-torus 
and $Y$ be a complex $2$-torus, Then, we have $\lambda_1(f) = \lambda_1(g)$ for $X=Y\times E$ and $f=(g,id_E) \in {\rm Aut}\, (X)$, and the result follows. 

Proof of ${\it Sal}(3) \subset {\it Sal}(2)$: Assume that $\lambda_1(f) \in {\it Sal}(3)$ for an automorphism $f$ of a complex $3$-torus $X$. Then, by Theorem 
\ref{TheoremToriAgain}, 
$X$ admits an $f$-equivariant holomorphic fibration $\pi : X \rightarrow Y$ 
over complex $2$-tori $Y$ 
such that $\lambda_1(g) = \lambda(f)$ for the induced automorphism $g$ of $Y$. 
This implies the result.
\end{proof}

\section{Special polynomials of degree $6$ - Geometry and Algebra}

In this section we introduce the notion of a {\it special polynomial} of degree $6$ 
and study their properties in both geometric representation of automorphisms of complex $3$-tori and algebraic representation of the Galois group. This is crucial in our proof of Theorems \ref{TheoremExamples} and 
\ref{TheoremPicardNumber} in Introduction.

\subsection{Special polynomials and complex $3$-tori}

Let $p(t)$ be a monic irreducible polynomial in $\mathbf{Z}[t]$ of degree $6$ 
such that $p(0)=1$. Let 
$${\mathcal S} := \{\zeta _1\,\, , \,\,\zeta _2\,\, , \,\,\zeta _3\,\, , \,\, \zeta _4\,\, ,\,\, \zeta _5\,\, ,\,\, \zeta _6\}$$ 
be the set of roots of $p(t)$ in ${\mathbf C}$. Here $\zeta_i$ 
are mutually distinct, $\zeta_i \not= \pm 1$, because $p(t)$ is irreducible 
in ${\mathbf Z}[t]$. We call $p(t)$ {\it special} if in addition the following 
condition (A) is satisfied (after re-ordering $\zeta_i$ suitably):

(A) $|\zeta _1|=|\zeta _2|=1$, $|\zeta _3|=|\zeta _6|>1$ and $|\zeta _4|=|\zeta _5|<1$ and all $\zeta_i$ are non-real.

\begin{remark}
1) $|\zeta _1|=|\zeta _2|=1$ already implies $\zeta_1$, $\zeta_2$ are non-real; otherwise some of them is $\pm 1$, a contradiction to the irreducibility of 
$p(t)$ in ${\mathbf Z}[t]$. 

2) $\zeta_2 = \overline{\zeta_1}$, $\zeta_6 = \overline{\zeta_3}$ and $\zeta_4 = \overline{\zeta_5}$ because $p(t)$ is of real coefficients so that if non-real $\zeta_i$ is a root of $p(t)$, then so is $\overline{\zeta_i}$. By $\zeta_2 = \overline{\zeta_1}$ and $|\zeta _1|=1$, we have $\zeta_1\zeta_2 = 1$. This will be very important in the next subsection. 

3) Let $f$ be an automorphism of a complex $3$-tori $X$. If $\lambda_1(f) = \lambda_2(f) > 1$ and $X$ has no $f$-equivariant holomorphic fibration, then the characteristic polynomial $\Phi(t)$ of $f_* \vert L$ is a special polynomial by Lemma \ref{basic} 3) and Theorem \ref{DoubleFibration} 2). Note that the fact that all roots are non-real also follows from Theorem \ref{DoubleFibration} 2). 
\label{Special}\end{remark}

Let $p(t)$ be a special polynomial with the roots $\zeta_i$ satisfying condition (A). We re-name $\zeta_i$ so that
$$\beta := \zeta_1\,\, ,\,\, \alpha := \zeta_3\,\, ,\,\, \gamma := \zeta_5\,\, .$$
Then $\zeta_2 = \overline{\beta}$, $\zeta_4 = \overline{\alpha}$ and 
$\zeta_6 = \overline{\gamma}$.

\begin{proposition}
Given a special polynomial $p(t)$, there are complex $3$-tori $X = {\mathbf C}^3/L$ and its automorphism $f$ such that the characteristic polynomial of $f_* \vert L$ is $p(t)$, $\lambda_1(f) = \lambda_2(f) >1$ and $X$ has no $f$-equivariant holomorphic fibration.
\label{Standard}\end{proposition}

\begin{proof} Let us consider the quotient ${\mathbf Z}$-module:
$$M := {\mathbf Z}[t]/(p(t))\,\, .$$
$M$ is a free ${\mathbf Z}$-module of rank $6$ and we have an automorphism $F \in {\rm Aut}\, (M)$ induced by the natural multiplication of $t$:
$$c(t)\, {\rm mod}\, (p(t))\,\, \mapsto\, tc(t)\, {\rm mod}\, (p(t))\,\, .$$
Here we note that the constant term of $p(t)$ is $1$ so that $F$ is an automorphism of $M$. The minimal polynomial of $F$ is $p(t)$ by definition. Let $V(\delta)$ be the eigenspace of the ${\mathbf C}$-linear extension of $F$ of the eigenvalue $\delta \in {\mathcal S}$. Then ${\rm dim}\, V(\delta) = 1$ and 
$$M \otimes {\mathbf C} = (V(\alpha) \oplus V(\beta) \oplus V(\gamma)) 
\oplus \overline{(V(\alpha) \oplus V(\beta) \oplus V(\gamma))}\,\, .$$
Since none of $\delta$ is real by (A), it follows that 
$$M/ (M \cap \overline{(V(\alpha) \oplus V(\beta) \oplus V(\gamma))}) \simeq M$$ is discrete and of rank $6$ in 
$$(M \otimes {\mathbf C})/
\overline{(V(\alpha) \oplus V(\beta) \oplus V(\gamma))} \simeq (V(\alpha) \oplus V(\beta) \oplus V(\gamma))\,\, .$$
Hence 
$$X := (M \otimes {\mathbf C})/(\overline{(V(\alpha) \oplus V(\beta) \oplus V(\gamma))} + M) \simeq (V(\alpha) \oplus V(\beta) \oplus V(\gamma))/M$$
is a complex $3$-torus. By construction, the action of $F$ on $M$ naturally descends to an automorphism $f$ of $X$ with $f_* = F$ on $H_1(X, {\mathbf Z}) = M$. For this $f$, we have $\lambda_1(f) = \vert \alpha \vert^2$, $\lambda_2(f) =  \vert \alpha \vert^2 \vert\beta \vert^2$ and therefore $1 < \lambda_1(f) = \lambda_2(f)$. Since $p(t)$ is irreducible in ${\mathbf Z}[t]$, $X$ has no $f$-equivariant holomorphic fibration by Theorem \ref{DoubleFibration} 2).
\end{proof}

\begin{remark}
1) In general, there are several ways to construct $(X, f)$ in Proposition \ref{Standard}. We call the construction in the proof a {\it standard construction} from $p(t)$. 
We note that any $(X, f)$ satisfying the condition of Proposition \ref{Standard} is isogenous to the standard one from $p(t)$, up to the $8$ possible choices of $\alpha$, $\beta$, $\gamma$ from ${\mathcal S}$. In fact $L = H_1(X, {\mathbf Z})$ is 
torsion free $R$-module of rank $1$. Here the ring $R$ is defined by 
$$R = {\mathbf Z}[t]/(p(t)) \simeq {\mathbf Z}[f_{*}] \subset 
{\rm End}_{\mathbf Z}\, (L)\,\, .$$ 
Hence $L \otimes{\mathbf Q }$ is isomorphic to 
$$M \otimes {\mathbf Q} = {\mathbf Q}[t]/(p(t)) \simeq {\mathbf Q}[f_{*}]$$ as ${\mathbf Q}[t]/(p(t))$-modules. This is because ${\mathbf Q}[t]/(p(t))$ is a field. Hence $(X, f)$ is isogenous to one of the $8$ standard ones. If in addition $R$ is PID, then $L$ is isomorphic to ${\mathbf Z}[t]/(p(t))$ as $R$-modules, so that $(X, f)$ is isomorphic to one of the $8$ standard ones.  

2) Geometric properties of $(X, f)$, for instance, projectivity, 
Picard number, are not so apparent even for the standard ones from $p(t)$. In order to clarify this, we need to study the Galois groups of special polynomials. This is the main topics of the next subsection. Note that projectivity, 
Picard number are invariant under isogenies. So, we may study them only for standard ones. 
\label{StadardRemark}\end{remark}

\subsection{Galois group of some special polynomials of degree $6$}
Let $p(t)$ be a special polynomial with the roots $\zeta_i$ with condition (A).
In this subsection we study the Galois group $G$ of the splitting field 
$$K = {\mathbf Q}(\zeta _1,\zeta _2,\zeta _3, \zeta _4, \zeta _5,\zeta _6)$$
of $p(t)$ over ${\mathbf Q}$. $G$ naturally acts on the set of roots 
${\mathcal S}$ and it is transitive by the irreducibility of $p(t)$. 
Hence $G$ is a transitive subgroup of 
$S_6 = {\rm Aut}_{\rm set}({\mathcal S})$.  

In addition to the condition (A), which is always assumed, we also consider the following conditions occasionally: 

(AR) $\zeta _3\zeta _4=\zeta _5\zeta _6=1$.  

(AP) $\zeta _1\zeta _3\zeta _5=1$. 

(AP), "{\it projectivity assumption}", is a condition closely related to 
the fact in Lemma \ref{basic} 5). 

The condition (AR), "{\it reciprocal assumption}", is an algebraic condition. 
Recall that $\zeta_1\zeta_2 = 1$. Hence (AR) implies that $p(t)$ is {\it reciprocal}, i.e., $p(t)$ is written as
\begin{equation}
p(t)=t^3q(t+\frac{1}{t})\,\, ,
\label{EquationReciprocal}\end{equation}
$q(t)$ is unique determined from $p(t)$ with (AR), and it is a monic polynomial of degree $3$, necessarily in ${\mathbf Z}[t]$. Since $p(t)$ is irreducible in ${\mathbf Z}[t]$, so is $q(t)$. We call $q(t)$ the {\it special trace polynomial} of a special polynomial $p(t)$ with (AR). Later, we will see that (AP) implies (AR) for special polynomials. 
\begin{lemma}
Assume the condition (AR) and define $x=\zeta _1 +\zeta _2$, $y=\zeta _3 +\zeta _4 $ and $z=\zeta _5+\zeta _6$. Then $q(t)=(t-x)(t-y)(t-z)$, $x$ is real, 
$\vert x \vert < 2$, $x \not= 0, \pm 1$ but $y$ and $z=\overline{y}$ are non-real. Moreover, the Galois group of $q(t)$, i.e., the Galois group of the field extension ${\mathbf Q}(x, y, z)/{\mathbf Q}$ 
is $S_3 = {\rm Aut}_{{\rm set}}\, (\{x, y, z\})$. 
\label{trace}
\end{lemma}

\begin{proof} By the definition of $q(t)$, we have $q(t)=(t-x)(t-y)(t-z)$. 
Since $\vert\zeta_1\vert = 1$ and $\zeta_2 = \overline{\zeta_1}$, $\zeta_1 \not= \pm 1$, it follows that $x$ is a real number with $-2 < x < 2$. 
If $x = 0, \pm 1$, then $\zeta$ would be a root of unity of degree $\le 2$, a contradiction to the irreducibility of $p(t)$ in ${\mathbf Z}[t]$. By definition $z = \overline{y}$. Since $q(t)$ is irreducible in ${\mathbf Z}[t]$, it follows that $y \not= z = \overline{y}$, whence they are non-real. 
Since $q(t) \in {\mathbf Z}[t]$ is an irreducible cubic polynomial with non-real root, it follows that 
$$[{\mathbf Q}(x, y, z) : {\mathbf Q}] = [{\mathbf Q}(x,y,z) : {\mathbf Q}(x)][{\mathbf Q}(x) : {\mathbf Q}] \ge 2 \times 3 = 6 = \vert S_3 \vert\,\, ,$$
whence the Galois group is the whole $S_3$.
\end{proof}

\begin{theorem}
Assume (AR). Then the Galois group $G$ conjugates in the symmetric group $S_6 
= {\rm Aut}_{\rm set}({\mathcal S})$ to one of the following $4$ groups (Here and hereafter we denote $\zeta_i$ by $i$ in the description):

$G_{48} = \langle (12)\, ,\, (13)(24)\, ,\, (135)(246) \rangle$. This group has $48$ elements.

$G_{24} = \langle (12)\, ,\, (135)(246) \rangle$. It is a subgroup of $G_{48}$, has $24$ elements and is isomorphic to $S_4$.

$H_{24}$ which is another subgroup of $G_{48}$ having $24$ elements and isomorphic to $S_4$ but not conjugates to $G_{24}$ in $S_6$. For a description of this group please see \cite{hagedorn}. For our purpose in this paper we need to know only that $H_{24}$ contains the group $G_{12}$ below.

$G_{12} = \langle (135)(246)\, , \, (13)(24)\, , \, (12)(34)(56) \rangle$. This is a subgroup of $H_{24}$ and has $12$ elements.

$H_6 = \langle (135)(246)\, , \, (12)(36)(45) \rangle$. This group is 
isomorphic to the symmetric group $S_3$.

\label{TheoremGaloisGroupNonProjectiveTori}\end{theorem}
\begin{proof}

From the assumption (AR), we see that if $\sigma \in G$ then it permutes the sets $\{\zeta _1,\zeta _2\}$, $\{\zeta _3,\zeta _4\}$ and $\{\zeta _5,\zeta _6\}$. Therefore $G$ conjugates to an imprimitive transitive subgroup of the group 
$G_{48}$, the latter being the semi-direct product $S_2^{\oplus 3}\rtimes S_3$, called the wreath product of $S_3$ and $S_2$. From the classification of transitive subgroups of $S_6$, we find that $G$ conjugates with one of the following $8$ groups : $G_{48}$, $G_{24}$, $\Gamma _{12}$, $G_{12}$, $C_6$ (the cyclic group of order $6$), $S_3$, $\Gamma _{24}$, $H_{24}$. Here we follow the notations in the paper Hagedorn \cite{hagedorn}.

It can be checked that the discriminant 
\begin{eqnarray*}
\Delta =\prod _{i<j}(\zeta _i-\zeta _j)^2
\end{eqnarray*}  
is negative. Therefore by Theorem 3 in \cite{hagedorn}, $G$ can not conjugate to the groups $\Gamma _{12}$ and $\Gamma _{24}$. It remains to check that $G$ can not conjugate to the cyclic group $C_6$. This is easy to see since if it was so, then $G$ has $6$ elements and hence it must be the Galois group $S_3$ of the splitting field over $\mathbf{Q}$ of the special trace polynomial $q(t)$, hence can not be $C_6$.
\end{proof}

We next compute the orbits under the actions of the groups $S_3$, $G_{12}$ and $G_{24}$ of the set 
$$\wedge^2 \mathcal{S} := \{\zeta _i\zeta _j,~1\leq i\not= j\leq 6\}\,\, .$$Here the numbers appear in $\wedge^2 \mathcal{S}$ with multiplicity, for example the number $1$ appears $3$ times in $\wedge^2 \mathcal{S}$ as numbers $\zeta _1\zeta _2$, $\zeta _3\zeta _4$ and $\zeta _5\zeta _6$; 
$\zeta _1\zeta _2 = \zeta _3\zeta _4 = \zeta _5\zeta _6 = 1$.

\begin{lemma}
1) $\wedge^2 \mathcal{S}$ decomposes under the action of the group 
$$H_6 = \langle(135)(246),(12)(36)(45) \rangle \simeq S_3$$ into the following four 
$S_3$-orbits:
\begin{eqnarray*}
Z_1&=&\{\zeta _1\zeta _4,\zeta _2\zeta _5,\zeta _3\zeta _6\},\\
Z_2&=&\{\zeta _1\zeta _6,\zeta _2\zeta _3,\zeta _4\zeta _5\},\\
Z_3&=&\{\zeta _1\zeta _2,\zeta _3\zeta _4,\zeta _5\zeta _6\},\\
Z_4&=&\{\zeta _1\zeta _3,\zeta _1\zeta _5,\zeta _2\zeta _4,\zeta _2\zeta _6,\zeta _3\zeta _5,\zeta _4\zeta _6\}.
\end{eqnarray*}

2) $\wedge^2 \mathcal{S}$ decomposes under the action of the group 
$$G_{12}= \langle (135)(246), (13)(24), (12)(34)(56) \rangle$$ into the following three 
$G_{12}$-orbits: 
\begin{eqnarray*}
Z_1&=&\{\zeta _3\zeta _6,\zeta _5\zeta _2,\zeta _1\zeta _6,\zeta _4\zeta _5,\zeta _1\zeta _4,\zeta _2\zeta _3\},\\
Z_2&=&\{\zeta _1\zeta _3,\zeta _3\zeta _5,\zeta _2\zeta _4,\zeta _5\zeta _1,\zeta _4\zeta _6,\zeta _6\zeta _1\},\\
Z_3&=&\{\zeta _1\zeta _2,\zeta _3\zeta _4,\zeta _5\zeta _6\}.
\end{eqnarray*}

3) $\mathcal{S}$ decomposes under the action of the group 
$$G_{24}= \langle (12), (135)(246) \rangle$$ into the following two $G_{24}$-orbits: 
\begin{eqnarray*}
Z_1&=&\{\zeta _3\zeta _6,\zeta _5\zeta _2,\zeta _1\zeta _6,\zeta _4\zeta _5,\zeta _1\zeta _4,\zeta _2\zeta _3,\zeta _1\zeta _3,\zeta _3\zeta _5,\zeta _2\zeta _4,\zeta _5\zeta _1,\zeta _4\zeta _6, \zeta _6\zeta _1\},\\
Z_2&=&\{\zeta _1\zeta _2,\zeta _3\zeta _4,\zeta _5\zeta _6\}.
\end{eqnarray*}
\label{LemmaOrbitGroupActions}\end{lemma}
\begin{proof}  
This follows from a direct calculation. 
\end{proof}

\begin{theorem}
Assume the condition (AP). Then the condition (AR) is satisfied and the Galois group $G$ is exactly either  
$$G_{12}= \langle (135)(246), (13)(24), (12)(34)(56) \rangle$$ 
or 
$$H_6 = \langle (135)(246),(12)(36)(45) \rangle \simeq S_3$$
as the subgroup of $S_6 = {\rm Aut}_{\rm set}({\mathcal S})$ (without taking conjugate in $S_6$).   
\label{TheoremGaloisGroupProjectiveTori}\end{theorem}
\begin{proof} Since $G$ acts transitively on ${\mathcal S}$, there is $\sigma \in G$ such that $\sigma (\zeta _1)=\zeta _3$. Then $\sigma (\zeta _2)=\zeta _4$ or $\zeta _5$ by $\zeta_1\zeta_2 =1$. First we observe that $\sigma (\zeta _2)$ can not be $\zeta _5$. Otherwise,  
\begin{eqnarray*}
\zeta _3 \zeta _5 =\sigma (\zeta _1 )\sigma (\zeta _2)=1=\zeta _1\zeta _3\zeta _5 ,
\end{eqnarray*}
which would imply that $\zeta _1 =1$, a contradiction. Therefore $\sigma (\zeta _2)=\zeta _4$, and moreover $\zeta _3\zeta _4=1$ again by $\zeta_1\zeta_2 =1$. 
Since $p(0)=1$ we have $\zeta _1\zeta _2\zeta _3\zeta _4\zeta _5\zeta _6=1$, which combined with $\zeta _1\zeta _2=1$ and $\zeta _3\zeta _4=1$ implies also that $\zeta _5\zeta _6=1$.  Hence (AR) is satisfied. Now we can speak of the 
special trace polynomial $q(t)$ of $p(t)$. 

Since the splitting field ${\mathbf Q}(x, y, z)$ of $q(t)$ under the notation of Lemma \ref{trace}, is a subfield of $K$, we have a surjective group 
homomorphism
$$q : G \to {\rm Aut}_{{\rm set}}(\{\{1,2\}\, , \, \{3, 4\}\, , \{5, 6\}\,\}) \simeq S_3\,\, .$$
Here the target group is naturally regarded as the Galois group of the trace polynomial $q(t)$ under the identification $x = \{1, 2\}$, $y = \{3, 4\}$, $z = \{5, 6\}$, and it is generated by $a_1 = (xyz)$ and $a_2 = (yz)$. 

\begin{lemma}
If $\zeta_i\zeta_j\zeta_k =1$($i \not= j \not= k \not= i$), then $\{i,j,k\}$ 
is either $\{1,3,5\}$ or $\{2, 4, 6\}$.
\label{Parity}
\end{lemma} 
\begin{proof} By (AP), we have $\zeta_1\zeta_3\zeta_5 = 1$ and $\zeta_2\zeta_4\zeta_6 = 1$. If some two of $i, j, k$ would be in $\{\{1,2\}, \{3, 4\}\, \{5, 6\}\}$, then 
$\zeta_i\zeta_j\zeta_k$ is non-real by (AR). So, we may assume that $i \in \{1,2\}$, $j \in \{3, 4\}$, $k \in \{5, 6\}$. If $i$, $j$, $k$ were not the same parity, then there would be unique $l \in \{i, j, k\}$, say $2$, 
such that the parity of $l$ is different from other two. For instance if $l =2$, Then, $\zeta_2\zeta_3\zeta_5 = 1$. Combining this with $\zeta_1\zeta_3\zeta_5 = 1$, we would have $\zeta_1 = \zeta_2$, a contradiction. 
\end{proof}

Now consider $N := {\rm Ker}\, q$. Let $\sigma \in N$. Then $\sigma(\zeta_1) 
\in \{\zeta_1, \zeta_2\}$, $\sigma(\zeta_3) 
\in \{\zeta_3, \zeta_4\}$ and $\sigma(\zeta_5) 
\in \{\zeta_5, \zeta_6\}$. Thus by Lemma \ref{Parity} and by
$$\sigma(\zeta_1)\sigma(\zeta_3)\sigma(\zeta_5) = 
\sigma(\zeta_1\zeta_3\zeta_5) 
= \sigma(1) = 1$$, 
it follows that $\sigma = id$ or $\sigma = (12)(34)(56)$. Hence $\vert N \vert \le 2$ and if $\vert N \vert = 2$, then $N = \langle (12)(34)(56)\rangle$ and $\vert G \vert = 12$.

If $N = \{id\}$, then $G \simeq S_3$, and the unique lifts $\sigma_1$, $\sigma_2$ of $a_1$ and $a_2$ generate $G$. We have 
$\sigma_1(\zeta_1), \sigma_1(\zeta_2)
\in \{\zeta_3, \zeta_4\}$, $\sigma_1(\zeta_3), \sigma_1(\zeta_4) 
\in \{\zeta_5, \zeta_6\}$ and $\sigma_1(\zeta_5), \sigma_1(\zeta_6)  
\in \{\zeta_1, \zeta_2\}$. By Lemma \ref{Parity}, $\zeta_1\zeta_3\zeta_5 = 1$ 
and $\zeta_2\zeta_4\zeta_6 = 1$, it follows that $\sigma_1 = (135)(246)$ or $\sigma_1 = (145236)$. Since $\sigma_1$ is of order $3$, it follows that $\sigma_1 = (135)(246)$. Since the complex conjugation map $\tau = (12)(36)(54)$ in $G$ and it is one of the lift of $a_2$. Thus $\sigma_2 = \tau$ and we are done. 

Consider the case $N = \langle (12)(34)(56) \rangle$. Set $\sigma_0 := (12)(34)(56)$. There are two choices of the lift of $a_1$ and $a_2$ respectively. The are transformed by $\sigma_0$. As in the case above, the complex conjugation map $\tau = (12)(36)(54)$ in $G$ and it is one of the lift of $a_2$. the possible lift of $a_1$ are $(135)(246)$ and $(145236)$ and they are transformed by $\sigma_0$. Hence we choose $(135)(246)$ as one of the lift of $a_1$.  
Hence the group $G$ is generated by $(12)(34)(56)$, $(12)(36)(54)$, $(135)(246)$, as claimed. This completes the proof.
\end{proof}

\subsection{Projective examples}  

\begin{theorem}
There is a projective complex $3$-torus of Picard number $9$ having an automorphism $f \in {\rm Aut}\, X$ 
such that $\lambda_2 (f) = \lambda_1(f)>1$ but $X$ has no $f$-equivariant holomorphic fibration. Moreover if $p(t)$ is the characteristic polynomial of $f^*:H^1(X,\mathbb{Z})\rightarrow H^1(X,\mathbb{Z})$ and $G$ is the Galois group of the splitting field of $p(t)$ over $\mathbb{Q}$, then both cases $G=S_3$ and $G=G_{12}$ exist.
\label{TheoremProjective3Torus}
\end{theorem}

\begin{proof} 

We give two example satisfying $G=S_3$ and $G=G_{12}$.

1) Example 1: $G=S_3$. Let $z_1$, $z_2$, $z_3$ be the zeros of 
$$h(x) := t^3 - t -1\,\, ,\,\, \in {\mathbf Z}[x]\,\, .$$ 
Then $z_i$ and $1/z_i = z_i^2 -1$ are all algebraic integers. So $z_i/z_j$ ($i \not= j$) are also algebraic integers. $h(x)$ has one real root and two non-real roots. Call $z_1$ the real root, then $z_3$ is the complex conjugate of $z_2$ and $z_2 \not= z_3$.. Then $z_1 >1$ and in fact it is the smallest Pisot number. Hence none of $z_i/z_j$ ($i \not= j$) is real. By the relation of zeros and coefficients for cubic polynomials, we have $z_1z_2z_3 = 1$. Hence
$$\vert z_1 \vert > 1 > \vert z_2 \vert = \vert z_3 \vert > 0\,\, .$$
The discriminant of the cubic polynomial $h(x)$ is 
$$-4\cdot(-1)^3 - 27\cdot(-1)^2 = -23 < 0$$
and this is not a square of a rational number. It follows that 
the Galois group $G$ of $h(x)$ over ${\mathbf Q}$ is $S_3$, the symmetric group of degree $3$. In particular, $G$ doubly transitively acts on 
$\{z_1, z_2, z_3\}$. Hence $G$ acts transitively on 
$\{z_i/z_j \vert i \not= j\}$. 
Hence the polynomial
$$p(t) := \prod_{i \not= j} (t - \frac{z_i}{z_j})=t^6+3t^5+5t^4+5t^3+5t^2+3t+1$$
is irreducible. It can be checked easily that $p(t)$ has zeros $${\mathcal S} := \{\alpha\,\, ,\,\, \beta\,\, ,\,\, \gamma\,\, , \overline{\alpha}\,\, ,
\,\, \overline{\beta}\,\, ,\,\, \overline{\gamma}\,\,\}\,\, ,$$
such that none of them are real and
$$\vert \alpha \vert > \vert \beta \vert = 1 > \vert \gamma \vert\,\, .$$

It remains to show the projectivity of $X$. We order the zeros of $p(t)$ as $\beta ,\overline{\beta},\alpha ,\overline{\gamma},\gamma ,\overline{\alpha}$ such that the condition $(AP)$ is satisfied. Let $K$ be the minimal splitting field of $p(t)$. We first show that the Galois group $G$ of $K/{\mathbf Q}$ is the group $S_3$ given in Theorem \ref{TheoremGaloisGroupProjectiveTori}. In fact, by Theorem \ref{TheoremGaloisGroupProjectiveTori} it suffices to show that $G$ is not the group $G_{12}$. Assume otherwise that $G=G_{12}$. Then by the computation in Lemma \ref{LemmaOrbitGroupActions}, $\{\alpha \overline{\alpha},\alpha \overline{\beta},\overline{\alpha}\beta , \beta \overline{\gamma},\gamma \overline{\beta},\gamma \overline{\gamma} \}$ is a $G$-orbit. Note that all of the numbers in $\{\alpha \overline{\alpha},\alpha \overline{\beta},\overline{\alpha}\beta , \beta \overline{\gamma},\gamma \overline{\beta},\gamma \overline{\gamma} \}$ are distinct: First, by using the absolute values of these numbers we see that the $4$ numbers $\alpha \overline{\alpha},\alpha \overline{\beta}, \beta \overline{\gamma},\gamma \overline{\gamma}$ are distinct, and then it follows easily that all the $6$ numbers are distinct since they are in a $G$-orbit. We conclude that the minimal polynomial of $\alpha \overline{\alpha}$ has degree $6$. But this gives a contradiction since $\alpha \overline{\alpha}=z_1^3$, and the minimal polynomial of $z_1^3$ has degree $3$. Therefore we showed that $G=S_3$. 

The computation in Lemma \ref{LemmaOrbitGroupActions} shows that $\{\alpha \overline{\alpha},\alpha \overline{\beta},\overline{\alpha}\beta\} , \{\beta \overline{\gamma},\gamma \overline{\beta},\gamma \overline{\gamma} \}$ are $G$-orbits.

Consider the natural action of $f$ and the $G$-action on
$$H^2(X, K) = H^2(X, {\mathbf Z}) \otimes K = (\wedge^2 H_1(X, {\mathbf Z})) 
\otimes K)^* = \wedge^2 ((H_1(X, {\mathbf Z}) \otimes K))^*$$
naturally induced by the action on $K$. Note that the eigenvalues of $f$ are in $K$. We denote by $W_K(\delta) \subset H^2(X, K)$, the eigenspace of the action of $f$ with eigenvalues $\delta$ over $K$. Note that 
$$\beta \overline{\beta}=\alpha \overline{\gamma}=\overline{\alpha}\gamma=1\,\, .$$
Then 
$$W = (W_K(\alpha \overline{\alpha}) \oplus W_K(\gamma \overline{\beta}) \oplus W_K(\beta\overline{\gamma}) \oplus W_K(\gamma \overline{\gamma}) \oplus W_K(\alpha \overline{\beta}) \oplus W_K(\overline{\alpha}\beta)) \oplus W_K(1)$$
is $f$-stable $K$-linear subspace of $H^2(X, K)$ of dimension $9$. (To see this we only need to check that none of the numbers $\alpha \beta ,\alpha \gamma , \beta \gamma$, $\overline{\alpha}\overline{\beta},\overline{\alpha}\overline{\gamma},\overline{\beta}\overline{\gamma}$ can be of the form $\zeta _i\overline{\zeta _j}$ for some $\zeta _i,\zeta _j\in \{\alpha ,\beta ,\gamma\}$. For example, if $\alpha \beta$ has this form, then by considering the absolute values, we conclude that $\alpha\beta$ must be either $\alpha \overline{\beta}$ or $\overline{\alpha}\beta$. But none of these are possible since both $\alpha$ and $\beta$ are non-real.)

By the description of the action of $G$ on $K$ in Theorem \ref{TheoremGaloisGroupProjectiveTori}, $W$ is also $G$-stable $K$-linear subspace of $H^2(X, K)$. Hence, by the theory of Galois descent (see e.g. Gille-Szamuelly \cite{gille-szamuelly}), there is a ${\mathbf Q}$-linear subspace $Q$ of $H^2(X, {\mathbf Q})$ such that 
$W = Q \otimes K$. Let $N := Q \cap H^2(X, {\mathbf Z})$. Then $N$ is a primitive ${\mathbf Z}$-submodule of rank $9$ of $H^2(X, {\mathbf Z})$ such that 
$$W = N \otimes K\,\, .$$
Then we have 
$$N \otimes {\mathbf C} = W \otimes {\mathbf C} = (W(\alpha \overline{\alpha}) \oplus W(\gamma \overline{\beta}) \oplus W(\beta\overline{\gamma}) \oplus W(\gamma \overline{\gamma}) \oplus W(\alpha \overline{\beta}) \oplus W(\overline{\alpha}\beta)) \oplus W(1)\,\, ,$$
where $W(\delta) \subset H^2(X, {\mathbf C})$ is the eigenspace of $f$ with eigenvalue $\delta$ over ${\mathbf C}$. 

Recall that our $X$ is $X = (V(\alpha) \oplus V(\beta) \oplus V(\gamma))/L$. 
Let $z_1$, $z_2$, $z_3$ be the global linear coordinate of $X$ corresponding 
to $V(\alpha)$, $V(\beta)$, $V(\gamma)$ respectively. Then, by the description of the action of $f$ on $H_1(X, {\mathbf Z})$, we see that each direct summand of the above decomposition is in $H^{1, 1}(X, {\mathbf C})$. For instance,  
$$W(\alpha \overline{\alpha}) = {\mathbf C} dz_1 
\wedge d\overline{z}_1\,\, ,$$ 
$$W(1) = {\mathbf C} \langle dz_2 \wedge d\overline{z}_2\,\, ,\,\, dz_1 \wedge d\overline{z}_3\,\, ,\,\, dz_3 \wedge d\overline{z}_1\, \rangle\,\, .$$
Hence
$$N \otimes {\mathbf C} = W \otimes {\mathbf C} = H^{1,1}(X, {\mathbf C})$$
by ${\rm rank}\, N = {\rm dim}\, H^{1,1}(X, {\mathbf C}) = 9$. Thus, again by the Galois descent applied for the extension ${\mathbf C}/{\mathbf R}$, we have
$$N \otimes {\mathbf R} = H^{1,1}(X, {\mathbf R})\,\, .$$
On one hand, the K\"ahler cone in $H^{1,1}(X, {\mathbf R})$ is open. On the other hand, $N \otimes {\mathbf Q}$ is dense in $N \otimes {\mathbf R}$. Thus, by the equality above, there is an element of $N$ which is also in the K\"ahler cone of $X$, that is, there is an integral K\"ahler class of $X$. Hence by the fundamental theorem of Kodaira, $X$ is projective as claimed. 

2) Example 2: We let $q(t)=t^3-5t^2+10t-1$. Using Maple, we can check that the discriminant of $q(t)$ is $<0$, $q(2)q(-2)<0$, and $q(0)q(1)q(-1)\not= 0$. Let $p(t)=t^3q(t+\frac{1}{t})$ $=t^6-5t^5+13t^4-11t^3+13t^2-5t+1$. Then $p(t)$ is irreducible and all of its solutions are non-real. By Maple we can check that the Galois group of $p(t)$ has order $12$. We can also check that $p(t)$ has $3$ roots $\alpha ,\beta ,\gamma$ with $|\alpha |>|\beta |=1>|\gamma |$ and $\alpha \beta \gamma =1$. If we order the roots of $p(t)$ as $\beta ,\overline{\beta},\alpha ,\overline{\gamma},\gamma ,\overline{\alpha}$, then Theorem \ref{TheoremGaloisGroupProjectiveTori} gives that $G$ is exactly the group $G_{12}$. Lemma \ref{LemmaOrbitGroupActions} shows that $\{\alpha \overline{\alpha},\alpha \overline{\beta},\overline{\alpha}\beta ,\beta \overline{\gamma},\overline{\beta}\gamma ,\gamma \overline{\gamma}\}$ is a $G$-orbit. If we construct the torus $X$ using the roots $\alpha ,\beta ,\gamma $ of $p(t)$, then we can show that $X$ is projective and has Picard number $9$ and the induced automorphism $f:X\rightarrow X$ satisfies the desired properties.  

\end{proof}

Conversely, if $X$ is a projective $3$-torus having $f:X\rightarrow X$ is an automorphism with $\lambda _1(f)=\lambda _2(f)>1$ but $f$ has no non-trivial equivariant holomorphic fibration, then the Picard number of $X$ must be $9$.
\begin{theorem}
Let $X$ be a projective $3$-torus having an automorphism $f:X\rightarrow X$ with $\lambda _1(f)=\lambda _2(f)>1$ but $f$ has no non-trivial equivariant holomorphic fibration. Then the Picard number of $X$ is $9$.
\label{TheoremPicardNumberProjective}\end{theorem}
\begin{proof}
By Proposition 14.5 in \cite{Ue}, after replacing $f$ by an iterate, we may assume that $f^*:H^{1,0}(X)\rightarrow H^{1,0}(X)$ has eigenvalues $\alpha ,\beta ,\gamma $ with $|\alpha |>|\beta |=1>|\gamma |$ and $\alpha \beta \gamma =1$. Then we list the roots of the characteristic polynomial $p(t)$ of $f^*:H^1(X,\mathbf{Z})\rightarrow H^1(X,\mathbf{Z})$ in the following order $\beta ,\overline{\beta},\alpha ,\overline{\gamma}, \gamma ,\overline{\alpha}$. By Theorem \ref{TheoremGaloisGroupProjectiveTori}, we see that the Galois group of the splitting field of $p(t)$ over $\mathbf{Q}$ is either $S_3$ or $G_{12}$. Then the argument in the proof of Theorem \ref{TheoremProjective3Torus} concludes that $X$ has Picard number $9$.
\end{proof}

\subsection{Non-projective examples}

We first show that the Picard numbers of non-projective examples must be $0$ or $3$, then show that these two cases can be realized.

\begin{theorem}
Let $X$ be a complex $3$-torus having an automorphism $f:X\rightarrow X$ with $\lambda _1(f)=\lambda _2(f)>1$ yet $f$ has no non-trivial equivariant holomorphic fibration. If $X$ is non-projective then the Picard number of $X$ is $0$ or $3$.
\label{TheoremPicardNumberNonProjective}\end{theorem}
\begin{proof}
Let $p(t)$ be the characteristic polynomial of $f^*:H^1(X,\mathbb{Z})\rightarrow H^1(X,\mathbb{Z})$. By Remark \ref{StadardRemark},  if we choose an appropriate order for the roots of $p(t)$, then $X$ is one of the examples constructed in Proposition \ref{Standard}.  Let $G$ be the Galois group over $\mathbf{Q}$ of $p(t)$. Then from Theorem \ref{TheoremGaloisGroupNonProjectiveTori}, $G$ is conjugate in $S_6$ to $H$, where $H$ is one of the groups $G_{48}$, $G_{24}$, $H_{24}$, $G_{12}$, $S_3$. Note that either $H=S_3$, or it contains $G_{12}$, or it contains $G_{24}$. Recall that $\tau =(\alpha \overline{\alpha})(\beta \overline{\beta})(\gamma \overline{\gamma})\in G$ is the complex conjugation.  

1) First consider the case $H=S_3$, hence $G$ is isomorphic to $S_3$. By Sylow's theorem, there is an element $\sigma$ of order $3$ of $G$, and $G$ is generated by $\sigma$ and the complex conjugation $\tau$. After composing on the left or right with the complex conjugation or replacing $\sigma $ with $\sigma ^2$, we may assume that $\sigma$ is one of the following $\sigma _1=(\beta \alpha \gamma )(\overline{\beta}\overline{\gamma}\overline{\alpha})$, $\sigma _2=(\beta \alpha \overline{\gamma})(\overline{\beta}\gamma \overline{\alpha})$, $\sigma _3=(\beta \alpha \overline{\alpha})(\overline{\beta} \gamma \overline{\gamma})$, and $\sigma _4=(\beta \alpha \overline{\alpha})(\overline{\beta}\overline{\gamma}\gamma )$.  

In the case $\sigma =\sigma _1$, we can argue as in Theorem \ref{TheoremProjective3Torus} to see that $X$ has Picard number $9$ and hence must be projective. But we assumed that $X$ is non-projective hence this case can not happen. 

In the case $\sigma =\sigma _2$, we find that the $G$-orbits are $Z_1=\{\alpha \overline{\alpha},\overline{\gamma}\overline{\beta},\beta \gamma\}$, $Z_2=\{\gamma \overline{\gamma},\beta \overline{\alpha},\overline{\beta}\alpha \}$, $Z_3=\{\alpha \beta ,\alpha \overline{\gamma},\beta \overline{\gamma},\overline{\alpha}\overline{\beta}, \overline{\alpha}\gamma ,\overline{\beta}\gamma \}$, and $Z_4=\{\beta \overline{\beta},\alpha \gamma ,\overline{\alpha}\overline{\gamma}\}$. From this, argue as in the proof of Theorem \ref{TheoremNonProjectiveTorus}, we find that the maximal $G$-stable set all of its numbers have the form $\lambda _i\overline{\lambda j}$ is $\{\gamma \overline{\gamma},\beta \overline{\alpha},\overline{\beta}\alpha ,\beta \overline{\beta}\}$. None of these $4$ numbers can belong to $Z_1$, $Z_2$ or $Z_3$. Hence $X$ has Picard number $\leq 4$. We next show that the Picard number of $X$ is exactly $3$. In fact, let $dz_1$, $dz_2$, $dz_3$, $d\overline{z_1}$, $d\overline{z_2}$, $d\overline{z_3}$ be the eigenvectors of $f^*:H^1(X,K)\rightarrow H^1(X,K)$ with eigenvalues $\beta$, $\alpha$, $\gamma$, $\overline{\beta}$, $\overline{\alpha}$, $\overline{\gamma}$. Since $\sigma =(\beta \alpha \overline{\gamma})(\overline{\beta}\gamma\overline{\alpha})$ it follows that $\sigma $ acts on the eigenvectors as follows $\{Kdz_1,Kdz_2,Kd\overline{z_3}\},~\{Kd\overline{z_1},Kdz_3,Kd\overline{dz_2}\}$. Therefore it follows that $\{Kdz_1\wedge d\overline{z_1},Kdz_2\wedge dz_3,Kd\overline{z_3}\wedge d\overline{z_2}\}$ is a $\sigma$-orbit, and in fact is a $G$-orbit (e.g. by considering the eigenvalues, which are all $1$). Hence $Kdz_1\wedge d\overline{z_1}$ is not in $NS_K(X)$, which implies that the Picard number of $X$ is $\leq 3$. It remains to show that the Picard number of $X$ is $\geq 3$. To see this we consider the $K$-vector space $W=W_K(\gamma \overline{\gamma})\oplus W_K(\beta \overline{\alpha} )\oplus W_K(\overline{\beta}\alpha )$. This vector space is $G$-stable, hence by the theory of Galois descend we have that there is $W'\in H^2(X,\mathbf{Q})$ such that $W=W'\otimes K$. Moreover, $W_{\mathbf{C}}=W\otimes _K\mathbf{C}$ is in $H^{1,1}(X)$. Therefore $W_{\mathbf{C}}\subset NS_{\mathbf{C}}(X)$. This implies that $NS_{\mathbf{C}}$ has Picard number at least $3$, and hence it has Picard number exactly $3$. 

The case $\sigma =\sigma _3$ can not happen, because $\sigma \circ \tau$ is the permutation $(\alpha \beta \gamma \overline{\beta}\alpha )$ thus can not belong to a group of order $6$. Similarly the case $\sigma =\sigma _4$ can not happen. 

Therefore in case 1) we have the Picard number of $X$ is $3$.

2) $G$ is not isomorphic to $S_3$. We consider two subcases:

a) Subcase 2.1: $\alpha \overline{\gamma}=\gamma\overline{\alpha}$. We order the roots as follows $\zeta _1=\beta$, $\zeta _2=\overline{\beta}$, $\zeta _3=\alpha $, $\zeta _4=\overline{\gamma}$, $\zeta _5=\gamma$, $\zeta _6=\overline{\alpha}$. 

We divide still into two further subcases:

Subcase 2.1.1: $G=\kappa H\kappa ^{-1}$ with $H<S_6$ contains $G_{12}$. In this case from Lemma \ref{LemmaOrbitGroupActions}, the $G_{12}$-orbits are

\begin{eqnarray*}
Z_1&=&\{\alpha\overline{\alpha},\gamma\overline{\beta},\beta\overline{\alpha},\overline{\gamma}\gamma,\beta\overline{\gamma},\overline{\beta}\alpha\},\\
Z_2&=&\{\beta\alpha,\alpha\gamma,\beta\gamma,\overline{\alpha}\overline{\beta},\overline{\alpha}\overline{\gamma},\overline{\beta}\overline{\gamma}\},\\
Z_3&=&\{\beta \overline{\beta}\}, Z_4=\{\alpha\overline{\gamma}\}, Z_5=\{\gamma \overline{\alpha}\}.
\end{eqnarray*}

Then any $G$-stable set $S$  must contain one of the following $\kappa Z_1$, $\kappa Z_2$, $Z_3$, $Z_4$, $Z_5$ (since the three numbers in $Z_3,Z_4,Z_5$ are $1$, these are $G$-orbits). Note that no element of $Z_1$ can belong to $Z_2$ and vice versa. Therefore no element of $\kappa Z_1$ can belong to $\kappa Z_2$ and vice versa. 

If both $\kappa Z_1$ and $\kappa Z_2$ belong to a same $G$-orbit, then argue as in 1) we see that $X$ has Picard number $3$. 

Consider now the case $\kappa Z_1$ and $\kappa Z_2$ are different $G$-orbits. If $\kappa Z_1\cap Z_2=\emptyset$ then $\kappa Z_1 =Z_1$, then using the $G$-stable set $Z_1,Z_3,Z_4,Z_5$ and arguing as in the proof of Theorem \ref{TheoremProjective3Torus}, we conclude that $X$ has Picard number $9$ and is projective which contradicts to our assumption. Similarly, if $\kappa Z_2\cap Z_2=\emptyset$ then we obtain a contradiction. Thus we must have $\kappa Z_i\cap Z_2\not= \emptyset$ for $i=1,2$. This implies that for any $G$-stable set $S$ which contains only numbers of the form $\lambda _i\overline{\lambda _j}$  must belong to the union of $Z_3,Z_4,Z_5$, and argue as in 1) we see that $X$ has Picard number $3$.

Subcase 2.1.2: $G=\kappa H\kappa ^{-1}$ with $H<S_6$ contains $G_{24}$. In this case we can argue as in the above to find that $X$ has Picard number $3$.

b) Subcase 2.2: $\alpha \gamma =\overline{\gamma}\overline{\alpha}=1$. Here we switch the roles of $\gamma$ and $\overline{\gamma}$ in the Subcase 2.1, and find the corresponding further subcases of Subcase 2.2

+ Subcase 2.2.1: $G=\kappa H\kappa ^{-1}$ with $H<S_6$ contains $G_{12}$. The $G_{12}$-orbits are 
\begin{eqnarray*}
Z_1&=&\{\alpha\overline{\alpha},\overline{\gamma}\overline{\beta},\beta\overline{\alpha},\overline{\gamma}\gamma,{\beta}{\gamma},\overline{\beta}\alpha\},\\
Z_2&=&\{\beta\alpha,\alpha\overline{\gamma},\beta\overline{\gamma},\overline{\alpha}\overline{\beta},\overline{\alpha}{\gamma},\overline{\beta}{\gamma}\},\\
Z_3&=&\{\beta \overline{\beta}\}, Z_4=\{\alpha{\gamma}\}, Z_5=\{\overline{\gamma} \overline{\alpha}\}.
\end{eqnarray*}

Again we find that $Z_1\cap Z_2=\emptyset$. Therefore $\kappa Z_1\cap \kappa Z_2=\emptyset$. Also, if $\kappa Z_1$ and $\kappa Z_2$ belong to the same $G$-orbit then we see that $X$ has Picard number $\leq 1$. Now we assume that $\kappa Z_1$ and $\kappa Z_2$ are different $G$-orbits. We can also assume that either $\kappa Z_1\cap \{\alpha \beta ,\overline{\alpha}\overline{\beta},\alpha \gamma , \overline{\alpha}\overline{\gamma}\}=\emptyset$ or $\kappa Z_2\cap \{\alpha \beta ,\overline{\alpha}\overline{\beta},\alpha \gamma , \overline{\alpha}\overline{\gamma}\}=\emptyset$, since otherwise if $S$ is a $G$-stable set whose any number is of the form $\lambda _i\overline{\lambda _j}$ then $S\cap \kappa Z_1=S\cap \kappa Z_2=\emptyset$, and hence $S$ has at most $1$ element $\beta \overline{\beta}$. Argue as in 1) we see that $X$ has Picard number $0$.

Thus we now consider two possibilities: 

i) Possibility 1: $\kappa Z_2\cap \{\alpha \beta ,\overline{\alpha}\overline{\beta},\alpha \gamma , \overline{\alpha}\overline{\gamma}\}=\emptyset$. It then follows that $\{\alpha \beta ,\overline{\alpha}\overline{\beta},\alpha \gamma , \overline{\alpha}\overline{\gamma}\}$ is contained in $\kappa Z_1$. Since any number in $\{\alpha ,\overline{\alpha},\beta ,\overline{\beta},\gamma ,\overline{\gamma}\}$ appears exactly twice as a factor of a number in $Z_1$, the same is true for $\kappa Z_1$. Thus we have two sub-possibilities:

-Sub-possibility 1.1: $\kappa Z_1=\{\alpha \beta ,\overline{\alpha}\overline{\beta},\beta \gamma ,\overline{\beta}\overline{\gamma}, \alpha \overline{\gamma},\overline{\alpha}\gamma\}$ and $\kappa Z_2=\{\alpha \overline{\alpha}, \overline{\alpha}\beta ,\overline{\gamma}\gamma ,\alpha \overline{\beta},\overline{\beta}\gamma ,\beta \overline{\gamma}\}$. We exclude this sub-possibility as follows. Let $\sigma \in G$ be an element of order $3$, which exists by Sylow's theorem. Since $\beta \overline{\beta}=\alpha \gamma =\overline{\alpha}\overline{\gamma}=1$, after composing $\sigma$ on the left and right with the complex conjugation as needed we can assume that $\sigma$ is either $\sigma _1=(\beta\alpha \overline{\alpha})(\overline{\beta}\gamma \overline{\gamma})$ or $\sigma _2=(\beta\alpha\overline{\gamma})(\overline{\beta}\gamma\overline{\alpha} )$. 

If $\sigma =\sigma _1$ then we find that $\sigma (\alpha \beta )=\alpha \overline{\alpha}$ which is a contradiction since $\alpha\beta$ and $\alpha \overline{\alpha}$ belong to different $G$-orbits. 

If $\sigma =\sigma _2$ then we find that $\sigma (\beta\gamma )=\alpha \overline{\alpha}$ and again obtain a contradiction.

Hence  sub-possibility 1.1 can not happen. 

-Sub-possibility 1.2: $\kappa Z_1=\{\alpha \beta ,\overline{\alpha}\overline{\beta},\beta \gamma ,\overline{\beta}\overline{\gamma}, \alpha \overline{\alpha},\overline{\gamma}\gamma\}$ and $\kappa Z_2=\{\alpha \overline{\gamma}, \overline{\alpha}\beta ,\overline{\alpha}\gamma ,\alpha \overline{\beta},\overline{\beta}\gamma ,\beta \overline{\gamma}\}$. We exclude this sub-possibility as follows. Let $\sigma \in G$ be an element of order $3$, which exists by Sylow's theorem. Since $\beta \overline{\beta}=\alpha \gamma =\overline{\alpha}\overline{\gamma}=1$, after composing $\sigma$ on the left and right with the complex conjugation as needed we can assume that $\sigma$ is either $\sigma _1=(\beta\alpha \overline{\alpha})(\overline{\beta}\gamma \overline{\gamma})$ or $\sigma _2=(\beta\alpha\overline{\gamma})(\overline{\beta}\gamma\overline{\alpha} )$. 

If $\sigma =\sigma _1$ we obtain the contradiction $\sigma (\alpha \overline{\alpha})=\overline{\alpha}\beta$, but these numbers belong to different $G$-orbits.

If $\sigma =\sigma _2$ we obtain the contradiction $\sigma (\gamma \overline{\gamma})=\overline{\alpha}\beta$. 

Hence sub-possibility 1.2 can not happen. 

ii) Possibility 2: $\kappa Z_1\cap \{\alpha \beta ,\overline{\alpha}\overline{\beta},\alpha \gamma , \overline{\alpha}\overline{\gamma}\}=\emptyset$. In this case we proceed as in Possibility 1. 

+ Subcase 2.2.2: $G=\kappa H\kappa ^{-1}$ with $H<S_6$ contains $G_{24}$. In this case we can argue as in the above to find that $X$ has Picard number $0$.
\end{proof}

\begin{theorem}
For $r=0$ or $3$, there is a non-projective $3$-torus $X$ of Picard number $r$ having an automorphism $f:X\rightarrow X$ with $\lambda _1(f)=\lambda _2(f)>1$, yet $f$ has no non-trivial equivariant holomorphic fibrations.
\label{TheoremNonProjectiveTorus}\end{theorem}
\begin{proof}
1) An example with Picard number $0$. Let $q(t)=t^3+t^2-1$. Since $-t^3q(1/t)=t^3-t-1$, it follows that $q(t)$ is irreducible, it has one real root $x$ and two non-real roots $y$ and $z=\overline{y}$. Moreover $|x|<2$ and $x\not= 0,\pm 1$. We let $$p(t)=t^3q(t+1/t)=t^6+t^5+3t^4+t^3+3t^2+t+1.$$

Then $p(t)$ has coefficients in $\mathbf{Z}$, $p(0)=1$, and we denote the roots of $p(t)$ by $\beta $ and $\overline{\beta}$ (the roots of $t^2-xt+1$), $\alpha ,\gamma $ (the roots of $t^2-yt+1$), and $\overline{\alpha},\overline{\gamma}$ (the roots of $t^2-zt+1$). We arrange such that $|\alpha |=|\overline{\alpha}|>|\beta |=|\overline{\beta}|=1>|\gamma |=|\overline{\gamma}|$. We can see that all of these numbers are non-real. Hence we can use the construction in Proposition \ref{Standard} for the polynomial $p(t)$ to obtain a complex $3$-torus $X$ and an automorphism $f:X\rightarrow X$ such that $\lambda _1(f)=\lambda _2(f)>1$, yet $f$ has no non-trivial equivariant holomorphic fibration. Using the software Maple we find that the Galois group $G$ of the splitting field over $\mathbf{Q}$ of $p(t)$ has order $48$. Proceed as in part 2b) of the proof of Theorem \ref{TheoremPicardNumberNonProjective}, we conclude that the Picard number of $X$ is $0$.

2) An example with Picard number $3$. We choose again the polynomial 
$$p(t) := t^6+3t^5+5t^4+5t^3+5t^2+3t+1$$
in  the proof of Theorem \ref{TheoremProjective3Torus}. Let again the zeros of $p(t)$ be $${\mathcal S} := \{\alpha\,\, ,\,\, \beta\,\, ,\,\, \gamma\,\, , \overline{\alpha}\,\, ,
\,\, \overline{\beta}\,\, ,\,\, \overline{\gamma}\,\,\}\,\, ,$$
such that
$$\vert \alpha \vert > \vert \beta \vert = 1 > \vert \gamma \vert\,\, .$$
Unlike in the proof of Theorem \ref{TheoremProjective3Torus},  we order the zeros of $p(t)$ as $\beta ,\overline{\beta},\alpha ,\overline{\gamma},\gamma ,\overline{\alpha}$ such that the condition $(AP)$ is {\it not} satisfied, i.e. here we order such that $\beta \alpha \gamma \not=1$. In this case, the Galois group of $K/\mathbf{Q}$ is not exactly $S_3$ but only isomorphic to $S_3$. The construction in Proposition \ref{Standard} for the polynomial $p(t)$ gives again a complex $3$-torus $X$ and an automorphism $f:X\rightarrow X$ such that $\lambda _1(f)=\lambda _2(f)>1$, yet $f$ has no non-trivial equivariant holomorphic fibration. In this case, part 1) of the proof of  Theorem \ref{TheoremPicardNumberNonProjective} gives that the Picard number of $X$ is $3$.

\end{proof}


\end{document}